\documentclass[a4paper, 12pt]{amsart}
\usepackage{amsmath,latexsym,amssymb}
\usepackage[hidelinks]{hyperref}

\newif\ifpdf
\ifx\pdfoutput\undefined
\pdffalse 
\else
\pdfoutput=1 
\pdftrue
\fi

\ifpdf
\usepackage[pdftex]{graphicx}
\else
\usepackage{graphicx}
\fi

\def\ordpair#1{(#1\nobreak)}
\def\rel#1#2{R_{#1}^{#2}}
\def\arity#1{a(\rel{}{#1})}

\def\dom{\mathop{\mathrm{dom}}\nolimits}

\newtheorem{lemma}{Lemma}[section]
\newtheorem{theorem}[lemma]{Theorem}
\newtheorem{corollary}[lemma]{Corollary}
\newtheorem{proposition}[lemma]{Proposition}
\newtheorem{conjecture}[lemma]{Conjecture}

\theoremstyle{definition}
\newtheorem{definition}[lemma]{Definition}

\newtheorem{question}[lemma]{Question}
\newtheorem{remark}[lemma]{Remark}
\newtheorem{remark*}{Remark}
\newtheorem{remarks}[lemma]{Remarks}
\newtheorem{example}[lemma]{Example}
\newtheorem{notation}[lemma]{Notation}

\def\N{\mathbb{N}}
\def\C{\mathcal{C}}
\def\D{\mathcal{D}}\def\E{\mathcal{E}}

\def\R{\mathbb{R}}

\def\Aut{{\rm Aut}}

\def\dom{{\rm dom}}

\def\a{\bar{a}}

\def\cl{{\rm cl}}
\def\scl{{\rm scl}}

\def\sdcl{{\rm sdcl}}

\def\S{\mathcal{S}}

\def\K{{\mathcal{K}}}

\def\sleq{\sqsubseteq_s}
\def\dleq{\sqsubseteq_d}

\def\sqleq{\sqsubseteq}
\def\Age{{\rm Age}}
\def\Roots{{\rm roots}}
\def\c{\bar{c}}
\def\Fraisse{Fra\"{\i}ss\'e}

\def\G{\mathcal{G}}
\def\str#1{{#1}}

\def\Ind#1#2{#1\setbox0=\hbox{$#1x$}\kern\wd0\hbox to 0pt{\hss$#1\mid$\hss}
\lower.9\ht0\hbox to 0pt{\hss$#1\smile$\hss}\kern\wd0}
\def\Notind#1#2{#1\setbox0=\hbox{$#1x$}\kern\wd0\hbox to 0pt{\mathchardef
\nn="3236\hss$#1\nn$\kern1.4\wd0\hss}\hbox to 0pt{\hss$#1\mid$\hss}\lower.9\ht0
\hbox to 0pt{\hss$#1\smile$\hss}\kern\wd0}

\usepackage{setspace}

\begin{document}

\title[]{Automorphism groups and Ramsey properties of sparse graphs}

\authors{

\author{David M. Evans}

\address{%
Department of Mathematics\\
Imperial College London\\
London\linebreak SW7~2AZ\\
UK.}

\email{david.evans@imperial.ac.uk}

\author{Jan Hubi\v cka}

\address{%
Department of Applied Mathematics (KAM)\\
Charles University\\
Malostrank\'e n\'am. 25\\
11800 Praha\\
Czech Republic.}

\email{hubicka@kam.mff.cuni.cz}

\author{Jaroslav Ne\v set\v ril}

\address{%
Computer Science Institute of Charles University (IUUK)\\
Charles University\\
Malostrank\'e n\'am. 25\\
11800 Praha\\
Czech Republic.}

\email{nesetril@iuuk.mff.cuni.cz}

\thanks{Jan Hubi\v cka is supported by project 18-13685Y of the Czech Science Foundation (GA\v CR)}

}

\date{15 January 2019}


\begin{abstract}  We study  automorphism groups of  sparse graphs from the viewpoint of topological dynamics and the Kechris, Pestov, Todor\v cevi\'c correspondence. We investigate amenable and ex\-treme\-ly amenable subgroups of these groups using the space of orientations of the graph and results from structural Ramsey theory. Resolving one of the open questions in the area, we show that Hrushovski's example of an $\omega$-categorical sparse graph has no $\omega$-categorical expansion with extremely amenable automorphism group.
\end{abstract}

\maketitle

\tableofcontents

\section{Introduction}

\subsection{Overview}
If $k$ is a natural number, we say that a graph is \emph{$k$-sparse} if, for every finite subgraph, the number of edges is bounded above by $k$ times the number of vertices. Classes of such graphs arise in model theory in Hrushovski's \emph{predimension constructions} and are an important source of counterexamples to many questions and conjectures in model-theoretic stability theory. The main aim of this paper is to study these classes from the twin viewpoints of structural Ramsey theory and topological dynamics. As we shall see,  Ramsey expansions of these classes exhibit rather different behaviour from that of classes studied previously and, correspondingly, the automorphism groups of the Fra\"{\i}ss\'e limits of
these classes exhibit new phenomena in topological dynamics. 
\medskip

The symmetric group $S_\infty$ (on a countably infinite set $M$)  can be considered as a Polish topological group by giving it the topology of pointwise convergence. With this topology, a subgroup of $S_\infty$ is closed if and only if it is the automorphism group of a first-order structure with domain $M$. A subgroup $G$ of $S_\infty$ is \emph{oligomorphic} if $G$ has finitely many orbits on $M^n$ for all natural numbers $n$. It is well known that, by the Ryll-Nardzewski Theorem, the closed, oligomorphic subgroups of $S_\infty$ are precisely the automorphism groups of $\omega$-categorical structures with domain $M$. 

Recall  that a topological group $G$ is \emph{extremely amenable} if whenever $X$ is a \emph{$G$-flow}, that is, a non-empty, compact Hausdorff $G$-space on which $G$ acts continuously, then there is a $G$-fixed point in $X$. The starting point for this paper is the following question, raised in \cite{BPT} and \cite[Question 1.1]{MTT}.

\begin{question}\label{qu} Suppose $G$ is a closed, oligomorphic permutation group on a countable set. Does there exist a closed, extremely amenable, oligomorphic subgroup  of $G$?
\end{question}

The question can be formulated in other ways. For example, it asks, given a countable $\omega$-categorical structure $M$, does there exist an $\omega$-categorical expansion of $M$ whose automorphism group is extremely amenable? Using the Kechris, Pestov, Todor\v cevi\'c 
(KPT for short) correspondence from \cite{KPT}, the question can be phrased in terms of \emph{Ramsey classes} and \emph{Ramsey lifts}, and in this form, it was asked by the third author as a question about the characterisation of Ramsey classes~\cite{NH}.
 More precisely, suppose $L$ is a first-order language and $\K$ is a Fra\"{\i}ss\'e class of finite $L$-structures. Thus there is a countable homogeneous $L$-structure with $\K$ as its class of (isomorphism types of) finite substructures. Suppose $L^+ \supseteq L$ is a language extending $L$. We say that a class $\K^+$ of finite $L^+$ structures is an
\emph{expansion} (\emph{lift} in \cite{HN}) of $\K$ if the $L$-reducts ($L$-shadow in \cite{HN}) of the structures in $\K^+$ form the class $\K$; it is a \emph{Ramsey expansion} (or \emph{Ramsey lift})
 if additionally it is a Ramsey class. The above question is then asking, in the case where $\K$ has only finitely many isomorphism types of structure of each finite size, whether there is a Ramsey expansion $\K^+$ of $\K$ with the same property.

\medskip

We discuss below some of the motivation for this question, but first we state some of our main results, showing that 
Question~\ref{qu} has a negative answer in general.

\begin{theorem} \label{cex} There exists a countable, $\omega$-categorical structure $M$ with the property that if $H \leq \Aut(M)$ is extremely amenable, then $H$ has infinitely many orbits on $M^2$. In particular, there is no $\omega$-categorical expansion of $M$ whose automorphism group is extremely amenable.
\end{theorem}

Using \cite{Z}, or \cite{BYMT}, we then have the following corollary, answering Question 1.5 in \cite{MTT}:

\begin{corollary} There is a closed, oligomorphic permutation group $G$ on a countable set whose universal minimal flow $M(G)$ is not metrizable.
\end{corollary}

As a direct corollary to the results in Section~\ref{sec:meagre}, we also show:

\begin{corollary} \label{cormeagre}There is a closed, oligomorphic permutation group $G$ which has a metrizable minimal flow where all $G$-orbits are meagre.
\end{corollary}

The example which gives these results is Hrushovski's construction of an $\omega$-categorical pseudoplane from \cite{Hrpp}. This is one of a variety of constructions of countable structures $M$ which we shall refer to as \emph{Hrushovski constructions}. Details will be given later in this paper (in Section~\ref{hcsec}; also Section~\ref{sec:C0}). One feature that all of the variations on the construction have is that $M$ interprets a {sparse} graph $\Gamma$. In this case, $\Aut(M)$ has a continuous action on the compact space $X_\Gamma$ of all
 \emph{orientations} of this graph (see Definition~\ref{or}). This is the main tool and object of study in this paper. In Section~\ref{sparsesec}, we describe $X_\Gamma$ and use it to prove Theorem~\ref{cex} (in the more general form of Theorem~\ref{sparsethm}).  As an additional benefit,
we also use it (in Section~\ref{nonambsec}) to give a simple proof of a general result (Theorem~\ref{38}) about non-amenability of $\Aut(M)$ which generalises results in \cite{GKP}. The argument we use shows that in Theorem~\ref{cex} we may take $M$ also having the property that there is no $\omega$-categorical expansion of $M$ with \emph{amenable}  automorphism group  (Corollary~\ref{noamenable}). In  Theorem~\ref{meagreorbits}, we give examples where every minimal subflow of $X_\Gamma$ has all orbits meagre, thereby proving Corollary~\ref{cormeagre}.

\medskip

The results of Kechris, Pestov and Todor\v cevi\'c in \cite{KPT} make a strong connection between the study of continuous actions of automorphism groups $G$ of countable structures on compact spaces (`topological dynamics') and Ramsey properties of classes $\K$ of finite structures (`structural Ramsey theory'). In particular, if $G = \Aut(M)$ preserves a linear ordering on $M$ and the language for $M$ is such that two tuples are in the same $\Aut(M)$-orbit iff they have the same quantifier free type (that is, $M$ is \emph{homogeneous}), then $G$ is extremely amenable if and only if the class $\K = \Age(M)$ of finite substructures of $M$ is a Ramsey class.

More generally, say that a subgroup $H \leq G = \Aut(M)$ is a \emph{co-precompact} subgroup of $G$ if, for every $n \in \N$, every $G$-orbit on $M^n$ is a union of finitely many $H$-orbits. If $H$ is closed, co-precompact in $G$ and extremely amenable, then (\cite{KPT, NVT}) the completion $\widehat{G/H}$ of the quotient space $G/H$ with respect to the right uniformity on $G$ is compact and the universal minimal flow $M(G)$ of $G$ is isomorphic to a minimal subflow of this. 
Thus, if one has a co-precompact, extremely amenable subgroup of $G$, then one has control over the universal minimal flow of $G$. Question~\ref{qu} asks whether one is guaranteed such a subgroup in the case where $M$ is $\omega$-categorical.

The above analysis shows that if $G$ has a co-precompact extremely amenable subgroup, then its universal minimal flow $M(G)$ is metrizable. In fact, the converse is also true: if $M(G)$ is metrizable, then there is a comeagre $G$-orbit on $M(G)$ and the stabilizer of a point in this orbit is extremely amenable and co-precompact in $G$. This is proved by Zucker in \cite{Z} and, independently, by Ben Yaacov, Melleray and Tsankov in \cite{BYMT}. (The latter builds on work in \cite{MTT} and proves the result for arbitrary Polish groups $G$.) 

Most proofs that a particular subgroup $H \leq \Aut(M)$ is extremely amenable make use of structural Ramsey theory. One identifies $H$ as the automorphism group of a homogeneous structure $N$ and shows that $\Age(N)$ is a Ramsey class (of ordered structures). Many examples of this can be found in the paper \cite{HN} and in the surveys \cite{B,N}.

The question of finding extremely amenable subgroups of $\Aut(M)$, or equivalently, finding good Ramsey expansions of $M$,  also has applications in the study of reducts of $M$ (see \cite{BP}) and hence to classifying constraint satisfaction problems with template $M$.  The paper \cite{Ivanov} by Ivanov also mentions the question of whether, if $\Aut(M)$ is amenable, then it has a precompact extremely amenable subgroup.

Our results show that the general problem of describing the universal minimal flow (and hence, all minimal flows) of a closed subgroup $G$ of $S_\infty$ is more complicated than the above picture suggests, even in the case where the subgroup is the automorphism group of an $\omega$-categorical structure (and therefore Roelcke precompact). In our examples, $G$ does not necessarily have the co-precompact extremely amenable subgroup needed to make the machinery work. Moreover, we show in Section 5 that for our examples, again in contrast to the above picture, minimal subflows of the space $X_\Gamma$ of orientations have all orbits meagre. So there are metrizable $G$-flows which have no comeagre orbit. 

It should be noted that Question~\ref{qu} remains open for $G = \Aut(M)$ where $M$ is a structure which is homogeneous for a finite relational language (the Hrushovski constructions require an infinite language for homogeneity).

In the Section~\ref{CFsec} we prove some positive results for our examples. We study two versions of the Hrushovski construction. The more technically challenging of these is the $\omega$-categor\-ical case considered in Section~\ref{CFsec}. Whilst this is perhaps the more important case from the point of view of Question~\ref{qu}, the other case is natural and of interest in its own right. In each case, we have an amalgamation class $(\C; \leq)$ of finite sparse graphs and a distinguished notion $\leq$ of `strong substructure'. There is an associated Fra\"{\i}ss\'e limit $M$ and, in each case, for $G = \Aut(M)$ we identify a maximal extremely amenable closed subgroup $H$ of $G$, corresponding to an `optimal' Ramsey expansion of $(\C; \leq)$ (Theorem~\ref{59}).

 These positive results raise the possibility that there might still be a weaker statement than the KPT correspondence in~\cite{NVT} which holds more generally. But in any case, the whole KPT-type correspondence for expansions is more complicated than perhaps was thought. The  Hrushovski construction leads to interpreting classes of structures which display a complicated behaviour and interplay of  related notions: Ramsey classes, the Expansion Property (Definition~\ref{EP:def}) and EPPA (Definition~\ref{def:EPPA}).

\medskip

\textit{Acknowledgements:\/} The first author thanks Todor Tsankov for numerous helpful discussions about the material in this paper, particularly about the proof of Theorem~\ref{38}. The authors also thank the Referee for numerous helpful suggestions and corrections.

\section{Background} For the convenience of the reader, we provide some background material on homogeneous structures, automorphism groups and Ramsey classes. Although we work with more general classes of finite structures than the \Fraisse~classes in, for example, \cite{KPT, NVT},  there is little that is new here and the reader who is familiar with this type of material can proceed to the following sections, referring back to this section where necessary.

We briefly review some standard model-theoretic notions.  Let
$L$ be a first-order relational language involving relational symbols $R\in L$
each having associated \emph{arities} denoted by $\arity{}$.  An 
\emph{$L$-structure} $\str{A}$ is a structure with \emph{domain} $A$, and
relations $R^A\subseteq A^{\arity{}}$, $R\in L$. The elements of the domain will often be referred to as \emph{vertices} of the structure.

The language is usually fixed and understood from the context (and it is in most cases denoted by $L$).  If the set $A$ is finite we call $\str A$ a \emph{finite structure}. We consider only structures with finitely or countably infinitely many vertices. 

A \emph{homomorphism} $f:\str{A}\to \str{B}$ is a mapping $f:A\to B$ such that  for every $R\in L$ we have
$(x_1,x_2,\ldots, x_{\arity{}})\in R^A\implies (f(x_1),f(x_2),\ldots,\allowbreak f(x_{\arity{}}))\in R^B$.
If $f$ is injective, then $f$ is called a \emph{monomorphism}. A monomorphism $f$ is an \emph{embedding} if the implication above is an equivalence.
  If $f$ is an embedding which is an inclusion then $\str{A}$ is a \emph{substructure}  of $\str{B}$. For an embedding $f:\str{A}\to \str{B}$ we say that $\str{A}$ is \emph{isomorphic} to $f(\str{A})$ and $f(\str{A})$ is also called a \emph{copy} of $\str{A}$ in $\str{B}$. 
  
 Finite structures will often be denoted by capital letters such as $A$, $B$, $C$ etc. and infinite structures by $M, N, \ldots$. We use the same notation for a structure and its domain and all substructures considered will be full  induced substructures. The automorphism group of a structure $M$ is denoted by $\Aut(M)$. 

\subsection{Fra\"{\i}ss\'e classes}\label{FCsec} Suppose $L$ is a first-order relational language. A countable $L$-structure $M$ is called \emph{(ultra)homogeneous} if isomorphisms between finite substructures extend to automorphisms of $M$. If $M$ is $\omega$-saturated, this is equivalent to the theory of $L$ having quantifier elimination. A homogeneous structure $M$ is characterised by its \emph{age}, the class $\Age(M)$ of isomorphism types of its finite substructures. This class satisfies the hereditary, joint embedding and (using the homogeneity) amalgamation properties. Conversely, if $\K$ is a class of countably many isomorphism types of finite $L$-structures which has these properties, then there is a countable homogeneous structure whose age is $\K$. All of this is the classical Fra\"{\i}ss\'e theory, initiated in \cite{Fr}. 

In one direction,  this result can be seen as a method for constructing homogeneous structures  from amalgamation classes of finite structures. There are several generalisations of this method and we shall state one of these, mostly following the presentation of Section 3 of \cite{BMPZ} and the notes \cite{DEBonn}. Essentially, we take \Fraisse's original construction, but instead of working with all substructures (equivalently, all embeddings between structures) we work with certain distinguished substructures, which we will call \emph{strong} substructures. Embeddings with strong substructures as their image will be called \emph{strong} embeddings.
Other generalisations are possible (though the basic structure of the proof is always the same). For example,  general category-theoretic versions of the Fra\"{\i}ss\'e construction can be found in \cite{DG} and \cite{Kubis}. 

\begin{definition} \rm \label{strongstrong}
Suppose $L$ is a first-order language and  $\K$ is a class of finite $L$-structures, closed under isomorphisms. Suppose $\S$ is a distinguished class  of embeddings between elements of $\K$. Assume that $\S$ is closed under composition and contains all isomorphisms. Furthermore, suppose that $(\K; \S)$ has the following property:
\begin{itemize}
\item[(*)]
whenever  $A, C \in \K$ and $f : A \to C$ is in $\S$ and $B \in \K$ is a substructure of $C$ which contains the image of $f$, then the  map $g : A \to B$ with $g(a) = f(a)$ for all $a \in A$ is in $\S$.
\end{itemize} 

Then we say that  $(\K; \S)$ is a \emph{strong class} and refer to the elements of $\S$ as \emph{strong embeddings}. If $A$ is a substructure of $B \in \K$ and the inclusion map $A \to B$ is in $\S$, then we say that $A$ is a \emph{strong substructure} of $B$ and write $A \leq B$. Thus an embedding $f : B \to C$ between structures in $\K$ is in $\S$ if and only if $f(B) \leq C$. Henceforth, we suppress the notation $\S$ and refer to the strong class as $(\K; \leq)$. We sometimes refer to the strong embeddings of this class as $\leq$-embeddings. In this notation the condition (*) says that if $A, B, C \in \K$ satisfy $A \leq C$ and $A \subseteq B \subseteq C$, then $A \leq B$.

If $\S$ consists of all embeddings between structures in $\K$, then we write $(\K; \subseteq)$ for the class, instead of $(\K; \leq)$.
\end{definition}

Note that if $A, B, C$ are in a strong class $(\K; \leq)$, then $A \leq A$ and $A \leq B \leq C$ implies that $A \leq C$. 

\begin{definition}\rm Suppose $(\K; \leq)$ is a strong class of finite $L$-structures and the $L$-structure $M = \bigcup_{i< \omega} A_i$ is the union of a chain of finite substructures $A_1 \leq A_2 \leq A_3 \leq \ldots$. If $A$ is a finite substructure of $M$, we write $A \leq M$ to mean that $A \leq A_i$ for some $i \leq \omega$, and say that $A$ is a strong substructure of $M$.
\end{definition}

\begin{remark}\rm It is important to note that the above definition does not depend on the choice of the sequence of $A_j$. Indeed, suppose also that $M$ is the union of the finite substructures $B_1 \leq B_2 \leq B_3 \leq \ldots$. Suppose $A \leq A_i$. There exist $j,k$ with  $A_i \subseteq B_j \subseteq A_k$; as $A_i \leq A_k$, the condition (*) implies that $A_i \leq B_j$ and so $A \leq B_j$. Note that this also shows that if $g \in \Aut(M)$ then $A \leq M$ if and only if $gA \leq M$. 

We also note that when we come to consider specific examples of strong classes, the definition of the strong embeddings will extend naturally to maps between arbitrary structures. We will usually omit the verification that the extension agrees with that in the above definition.
\end{remark}

\begin{definition} Suppose $(\K; \leq)$ is a strong class of finite $L$-structures. An increasing chain 
\[ A_0 \leq A_1 \leq A_2 \leq A_3 \leq \cdots\]
of structures in $\K$ is called a \emph{rich sequence} if: 
\begin{enumerate}
\item for all $A \in \K$ there is some $i < \omega$ and a strong embedding $A \to A_i$;
\item for all strong $f : A_i \to B$ there is $j \geq i$ and a strong $g: B \to A_j$ such that $g(f(a)) = a$ for all $a \in A_i$. 
\end{enumerate}
A \emph{Fra\"{\i}ss\'e limit} of $(\K; \leq)$ is an $L$-structure which is the union of a rich sequence of substructures.
\end{definition}

\begin{definition}  We say that the strong class $(\K; \leq)$ has the \emph{amalgamation property}  (for strong embeddings), \emph{AP} for short,
 if  whenever  $A_0, A_1, A_2$ are in $\K$ and $f_1: A_0 \to A_1$ and $f_2 : A_0 \to A_2$ are strong, there is $B \in \K$ and strong $g_i : A_i \to B$ (for $i = 1,2$) with $g_1\circ f_1 = g_2 \circ f_2$. The class has the \emph{joint embedding property}, \emph{JEP} for short,
 if for all $A_1, A_2 \in \K$ there is some $C \in \K$ and strong $f_1: A_1 \to C$ and $f_2: A_2 \to C$.
\end{definition}

If $\emptyset \in \K$ and $\emptyset \leq A$ for all $A \in \K$, then the JEP is a special case of the AP. 

\begin{theorem}\label{FHT} Suppose $(\K; \leq)$ is a strong class of $L$-structures which satisfies:
\begin{enumerate}
\item There are countably many isomorphism types of structures in $\K$.
\item The class $\K$ has the Joint Embedding and Amalgamation Properties with respect to strong embeddings.
\end{enumerate}
Then rich sequences exist and all Fra\"{\i}ss\'e limits are isomorphic. Moreover, if $M$ is a Fra\"{\i}ss\'e limit and $f : A \to B$ is an isomorphism between strong finite substructures of $M$, then $f$ extends to an automorphism of $M$.
\end{theorem}

We refer to the last property in the above as \emph{$\leq$-homogeneity} (or strong-map homogeneity) and say that the  strong class $(\K; \leq)$ is an \emph{amalgamation class} (for strong maps) if conditions (1, 2) hold. Note that in the case where  $\K$ is closed under substructures and $\S$ consists of all embeddings between structures in $\K$, this result is the classical Fra\"{\i}ss\'e Theorem.

\medskip

\begin{remarks}\label{freeamcl} \rm Many of the examples of amalgamation classes $(\K; \leq)$ of relational structures in this paper will be \emph{free amalgamation classes}. If $A \leq B_1, B_2$ are structures in $\K$, then by the \emph{free amalgam} of $B_1$ and $B_2$ over $A$ we mean the structure $F$ whose domain is the disjoint union of $B_1$ and $B_2$ over $A$ and in which the relations $R^F$ (for $R$ in the language) are just the unions $R^{B_1}\cup R^{B_2}$ of the the relations on $B_1$ and $B_2$. If $F$ is always in $\K$ and $B_i \leq F$, then we say that $(\K; \leq)$ is a \emph{free amalgamation class}.
\end{remarks}

\begin{remarks}\label{34} Suppose $\K$ in Theorem~\ref{FHT} has only finitely many isomorphism types of structure of each finite size. Suppose also that there is a function $F : \N \to \N$ such that if $B \in \K$ and $A \subseteq B$ with $\lvert A \rvert \leq n$, then there is $C \leq B$ with $A \subseteq C$ and $\lvert C \rvert \leq F(n)$. Then the Fra\"{\i}ss\'e limit $M$ is $\omega$-categorical. 

To see this we note that $\Aut(M)$ has finitely many orbits on $M^n$. Indeed, by $\leq$-homogeneity there are finitely many orbits on $\{\c \in M^{F(n)} : \c \leq M\}$ and any $\a \in M^n$ can be extended to an element of this set.
\end{remarks}

Much of this paper will be concerned with expansions of Fra\"{\i}ss\'e limits, or their corresponding amalgamation classes. The following notions will be useful.

\begin{definition} \label{stex:def} Suppose that $L \subseteq L^+$ are first-order languages and $(\K^+ ; \leq^+)$, $(\K; \leq)$ are strong classes of finite $L^+$- and $L$-structures respectively. We say that $(\K^+; \leq^+)$ is a \emph{strong expansion} of $(\K; \leq)$ if $\K$ is the class of $L$-reducts of the structures in $\K^+$ and:
\begin{enumerate}
\item[(i)] whenever $f: A^+ \to B^+$ is a strong map in $(\K^+; \leq^+)$, the map between the $L$-reducts $f : A \to B$ is strong in $(\K; \leq)$;
\item[(ii)] if $g : A \to B$ is a strong map in $(\K; \leq)$ and $A^+ \in \K^+$ is an expansion of $A$, then there is an expansion $B^+$ of $B$ such that $g: A^+ \to B^+$ is strong in $(\K^+; \leq^+)$.
\end{enumerate}
\end{definition}

\begin{theorem} \label{sexpthm} Let $L \subseteq L^+$ be first-order languages. Suppose that $(\K^+ ; \leq^+)$ is an amalgamation class of finite $L^+$-structures which is a strong expansion of the strong class $(\K; \leq)$ of $L$-structures. Then $(\K; \leq)$ is an amalgamation class. Moreover, if $N$, $M$ denote the Fra\"{\i}ss\'e limits of $(\K^+; \leq^+)$ and $(\K; \leq)$ respectively, then the $L$-reduct of $N$ is isomorphic to $M$.
\end{theorem}

\begin{proof} Suppose $f_i : A \to B_i$ are strong embeddings in $(\K; \leq)$, for $i = 1,2$. We can expand $A, B_1$ to structures $A^+, B_1^+$ in $\K^+$ so that $f_1 : A^+ \to B_1^+$ is strong. We can then expand $B_2$ to a structure $B_2^+$ in $\K^+$ so that $f_2 : A^+ \to B_2^+$ is strong. The amalgamation property in $(\K; \leq^+)$ gives that there exist $D^+ \in \K^+$ and strong $g_i : B_i^+ \to D^+$ with $g_1\circ f_1 = g_2\circ f_2$. Passing to the $L$-reducts, we obtain the amalgamation property for $(\K; \leq)$. So this is an amalgamation class.

Similarly, suppose $A_1^+ \leq^+ A_2^+ \leq A_3^+ \leq \cdots $ is a rich sequence for $(\K^+; \leq^+)$. Then the sequence of $L$-reducts $A_1\leq A_2 \leq A_3 \leq \cdots$ is easily seen to be rich for $(\K; \leq)$. The result follows.
\end{proof}

Examples of such strong expansions will be found in later sections (for instance, see Section~\ref{sec:C0}).

\subsection{Ramsey classes} \label{rcsec}

\medskip

Throughout this subsection, $L$ will be a first-order language and we work with strong classes $(\K; \leq)$ of finite $L$-structures as in Section~\ref{FCsec}. We shall say what it means for $(\K; \leq)$ to be a Ramsey class and in the next subsection,we state the KPT correspondence and associated results in this context. Similar statements (about special class of maps)
 can be found in the paper of Zucker \cite{Z} and in \cite{GKP}. In the case where $\K$ is closed under substructures and $\leq$ is just the usual notion of substructure, this is just the usual notion of Ramsey class and the KPT correspondence. The statements which we give can all be deduced from this case either by simple adaptations of  the proofs, or, more directly, by expanding the language in a suitable way. Indeed, the latter is the approach taken by the second and third authors in \cite{HN}, where the results are stated for classes of structures with closures. This latter paper is the main source of the Ramsey results which we will use in the later sections, so we will state its results in detail.

Suppose $L$ is a first-order language and that $(\K; \leq)$ is a strong class of finite $L$-structures. For $A, B \in \K$ we denote by $\binom{B}{A}$ the set of all \emph{strong} substructures of $B$ which are isomorphic to $A$. Note that by the condition (*) in Definition~\ref{strongstrong}, if $B \leq C \in \K$ and $A \in \K$, then $\binom{B}{A} = \binom{C}{A}\cap \mathcal{P}(B)$, that is, the strong copies of $A$ in $B$ are precisely the strong copies of $A$ in $C$ whose domains are subsets of $B$.

\smallskip

We say that $(\K; \leq)$ is a \emph{Ramsey class} if it is a strong class and for all $r \in \N$ and all $A, B \in \K$, there is a structure $C \geq B$ in $\K$ such that the following \emph{Ramsey property} holds: whenever $\binom{C}{A}$ is partitioned into $r$ classes (`colours'), there is $B' \in \binom{C}{B}$ such that the elements of $\binom{B'}{A}$ all lie in the same class (that is, they have the same colour). In this case, we write:
\[ C \longrightarrow (B)_r^A.\]
Note that here we are restricting to strong substructures throughout (without incorporating this into the notation) and it is of course enough to consider this in the case $r=2$. 

More generally, we say that $A \in \K$ has \emph{finite Ramsey degree} if there is a natural number $k$ such that for all $B \in \K$ with $A \leq B$ and all $r \in \N$,  there is $C \geq B$ in $\K$ such that whenever $\binom{C}{A}$ is coloured with $r$ colours, there is $B' \in \binom{C}{B}$ such that $\binom{B'}{A}$ is coloured with at most $k$ colours. The least such $k$ is then the \textit{Ramsey degree} of $A$ in $(\K; \leq)$. Note that if this is equal to 1 for all $A \in \K$, then $(\K; \leq)$ has the Ramsey property.

The Ramsey property is sometimes defined with respect to colourings of embeddings from $A$ to $B$ and $C$. If the structures in $\K$ are \emph{rigid} (that is, have trivial automorphism group), then there is no difference between these notions. This is the case if, for example, each structure in $\K$ has a linear ordering as part of the structure.

 In the case where all substructures are strong, it is a well known observation of the third author  (cf. \cite{N89}, for example) that (under mild extra conditions) Ramsey classes are amalgamation classes. We note that the usual argument also applies in our current context (of strong maps).

\begin{theorem} \label{RimpliesAP} Suppose that $L$ is a first-order language and $(\K; \leq)$ is a Ramsey class of finite, rigid $L$-structures with the joint embedding property. Then $(\K; \leq)$ has the amalgamation property.
\end{theorem}

\begin{proof} Let $f_i : A \to B_i$ be strong embeddings forming our amalgamation problem. As the structures are rigid, it is enough to find $E \in \K$ which contains strong copies of $B_1, B_2$ having a copy of $A$ as a common strong substructure. 

There is some $D \in \K$ which contains strong copies of $B_1$ and $B_2$ (using JEP). Find $E \in \K$ with $E \to (D)_2^A$. Colour the elements of $\binom{E}{A}$ according to whether or not they are contained in a strong copy of $B_1$ in $E$. There is a monochrome copy $D'$ of $D$. As it contains a strong copy of $B_1$, all the strong copies of $A$ in it are in a strong copy of $B_1$ in $E$. But this includes the $A$ which is in the copy of $B_2$ in $D'$.
\end{proof}

We now state, using this terminology, the general Ramsey result which we need. While the original \Fraisse{} construction (where all embeddings are strong) generalises naturally to strong amalgamation
classes, it is the essence of our examples to show that this is not the case for the construction
of Ramsey objects. For this reason, the papers~\cite{HN,EHN} use an alternative approach, representing strong substructures by means of functions which are part of the structures themselves instead of by an external family of strong embeddings. We review the main
terminology and results of \cite{EHN} (using several results of \cite{HN}) which will be needed here.

First we introduce a notion of structure involving functions in addition to
relational symbols. Unlike the usual model-theoretic functions, functions used
here are partial, multi-valued and symmetric. Partial functions allow the easy
definition of free amalgamation classes and the symmetry makes it possible to
explicitly represent strong embeddings within the structure itself, while
keeping all automorphisms of the original structure.

Let $L=L_\mathcal R\cup L_\mathcal F$ be a language involving relational
symbols $\rel{}{}\in L_\mathcal R$ and function symbols $F\in L_\mathcal F$
each having associated \emph{arities} denoted by $\arity{}$ for relations and
\emph{domain arity}, $d(F)$, \emph{range arity}, $r(F)$, for functions.  Denote
by $A\choose n$ the set of all subsets of $A$ consisting of $n$ elements.
An \emph{$L$-structure} $\str{A}$ is a structure with \emph{domain} $A$,
functions $F^A:\dom(F^A)\to {A\choose {r(F)}}$, $\dom(F^A)\subseteq A^{d(F)}$,
$F\in L_\mathcal F$ and relations $R^A\subseteq A^{\arity{}}$, $R\in
L_\mathcal R$.  The set $\dom(F)$ is called the \emph{domain} of the function $F$  in
$\str{A}$.

Given two $L$-structures $A$ and $B$, we say that $A$ is a \emph{substructure}
of $B$ and write $A\subseteq B$ if the following holds:
\begin{enumerate}
 \item the domain of $A$ is a subset of domain of $B$,
 \item for every relation $R\in L_\mathcal R$ it holds that $R^A$ is the restriction of $R^B$ to $A$, and,
 \item for every function $F\in L_\mathcal F$ it holds that $\dom(F^A)$ is the restriction of $\dom(F^B)$ to $A$
and moreover for every $t\in \dom(F^A)$ it holds that $F^A(t)=F^B(t)$.
\end{enumerate}
Embeddings are defined analogously (a substructure then expresses the fact that inclusion is an embedding).

If $A, B_1, B_2$ are $L$-structures and $\alpha_i : A \to B_i$ are embeddings, then an $L$-structure $C$ together with embeddings $\beta_i : B_i \to C$ is called an \emph{amalgamation} of $B_1$ and $B_2$ over $A$ if $\beta_1(\alpha_1(a)) = \beta_2(\alpha_2(a))$ for all $a \in A$. It is a \emph{free amalgamation} of $B_1$ and $B_2$ over $A$ if $\beta_1(b_1) = \beta_2(b_2)$  only if $b_i \in \alpha_i(A)$  and moreover there are no tuples in any relations $R^C$ of $C$ and no tuples in $\dom(F^C)$ (with $R \in L_\mathcal{R}$ and $F \in L_\mathcal{F}$)  using vertices of both $\beta_1(B_1\setminus \alpha_1(A))$ and $\beta_2(B_2\setminus \alpha_2(A))$, and $C = \beta_1(B_1)\cup \beta_2(B_2)$. 

In the case where $L$ consists only of relation symbols, note that this coincides with the usual notion of free amalgamation (as in Remarks~\ref{freeamcl}). 

Suppose now that $(\K_0; \leq)$ is a strong class of $L_0$-structures. Suppose, moreover, that strong substructures are closed under intersections (and therefore there is an associated notion of closure).  Then there is a standard way to turn this class into amalgamation class $(\K,\subseteq)$ which is closed for substructures:  We can expand $L_0$ to a language $L$ by adding partial functions $F_{k,n}$, for every $1\leq k\leq n$, from $k$-tuples to subsets of size $n$.  On a structure $A \in \K_0$, map every set of elements $S\subseteq A$ to the smallest strong substructure $B$ of $A$ containing $S$, that is $F_{|S|,|B|}(S)=B$ and leave the functions undefined otherwise. Note that doing this does not affect the automorphisms of $A$.  The resulting class $(\K; \subseteq)$ is then a strong class (where strong maps are all embeddings) and if $(\K_0; \leq)$ is an amalgamation class, then so is $(\K; \subseteq)$.

Observe that even if $(\K_0; \leq)$ is a free amalgamation class, then $(\K; \subseteq)$ constructed in this standard way is not necessarily a free amalgamation class. 
However in cases discussed here we will be able to omit some of the functions to obtain $(\K';\subseteq)$ which is closed for free amalgamation.
This will allow us, in Section~\ref{sec:C0sec} to apply the following theorem to show that such a class has an easy Ramsey expansion.

Given a class $\K$ of $L$-structures, denote by $\K^\prec$ the class
of all structures $\ordpair{A;\prec}$ where $A\in \K$ and $\prec$ is a linear
ordering of the domain of $A$. The following is a combination of Theorems 1.3 and 1.4
of \cite{EHN} which will be applied in Section~\ref{sec:DFramsey}. The \emph{Expansion Property} is defined in Definition~\ref{EP:def} below.
\begin{theorem}
\label{thm:Ramseyclosures}
Let $L$ be a language (involving relational symbols and partial functions) and let 
$(\K,\subseteq)$ be a free amalgamation class.
  Then $(\K^\prec,\subseteq)$ is a Ramsey class and moreover there exists a
Ramsey class $\mathcal O\subseteq \K^\prec$ such that 
$(\mathcal O,\subseteq)$ is a strong expansion of $(\K,\subseteq)$ having the
Expansion Property with respect to $(\K,\subseteq)$.
\end{theorem}

\subsection{The KPT correspondence} \label{KPTsec} The fundamental connection between Ramsey classes and topological dynamics is the following result of Kechris, Pestov and Todor\v cevi\'c
which we formulate in the following way in the context of strong maps. Recall that an amalgamation class $(\K; \leq)$ is a strong class satisfying the conditions (1,2) of Theorem~\ref{FHT}.

\begin{theorem}[\cite{KPT}, Theorem 4.8; \cite{NVT}, Theorem 1]
\label{thm:KPT}
 Let $L$ be a first-order language and $(\K; \leq)$ an amalgamation class of finite, rigid $L$-structures. Let $N$ be the Fra\"{\i}ss\'e limit of the class. Then $\Aut(N)$ is extremely amenable if and only if $(\K; \leq)$ is a Ramsey class.
\end{theorem}

Now we modify this for expansions.
\begin{definition} \label{def:reasonable} Suppose that $L \subseteq L^+$ are first-order languages and $L^+\setminus L$ consists of relation symbols. Let $(\K; \leq)$ be an amalgamation class of finite $L$-structures  with Fra\"{\i}ss\'e limit $M$. Suppose that $\D$ is a class of finite $L^+$-structures with the properties:
\begin{enumerate}
\item[(i)] the class of $L$-reducts of $\D$ is $\K$;
\item[(ii)] each structure in $\K$ has finitely many isomorphism types of expansions in $\D$;
\item[(iii)] if $B^+ \in \D$ and $A^+$ is a substructure of $B^+$ with $A^+\leq B^+$ (that is, the corresponding $L$-reducts $A, B$ satisfy $A \leq B$), then $A^+ \in \D$; 
\item[(iv)] if $f: A \to B$ is strong in $(\K; \leq)$ and $A^+\in \D$ is an expansion of $A$, then there is an expansion $B^+ \in \D$ of $B$ such that $f : A^+ \to B^+$ is an embedding.
\end{enumerate}
Then we say that  $\D$ as a \emph{reasonable class} of expansions of $(\K; \leq)$. 
\end{definition} 

The above terminology follows  \cite{Z}. Note however that we include (ii) as part of the definition rather than referring to it as `precompactness'. If $(\K^+; \leq^+)$ is a strong expansion of $(\K; \leq)$ as in Definition \ref{stex:def}, then $\K^+$ satisfies conditions (i) and (iv) in the above. If additionally $\K^+$ is a reasonable class of expansions of $(\K; \leq)$ (that is, it also satisfies (ii), (iii) in the above), then we refer to $(\K^+; \leq^+)$ as a \textit{reasonable, strong expansion} of $(\K; \leq)$. 

\medskip

 Suppose that $\D$ is a reasonable class of expansions of $(\K; \leq)$ as in the above definition and $M$ is the Fra\"{\i}ss\'e limit of $(\K; \leq)$. We shall consider the set $X(\D)$ of $L^+$-expansions $M^+$ of $M$ which have the property that for every finite $A \leq M$, the $L^+$-structure $A^{M^+}$ induced on $A$ in $M^+$ is in the class $\D$. 
This is a topological space where a basic open set is given by considering the expansions $M^+$ in which $A^{M^+}$ is a fixed structure in $\D$ (for a finite $A \leq M$). The property (ii) in the definition implies that $X(\D)$ is compact and it follows from properties (iii) and (iv) that $X(\D)$ is non-empty. Note that if  $L^+\setminus L$ consists of finitely many relation symbols $R_i$ of arities $n_i$ (for $i \leq m$), then $X(\D)$ can be identified with a subset of $\prod_{i\leq m} M^{n_i}$ with the product topology. In general,  $X(\D)$ embeds in an inverse limit of such spaces, by the  property (ii) in Definition~\ref{def:reasonable}.

With this notation, we have the following, summarising the above and statements in \cite{KPT,NVT}:

\begin{theorem} \label{XDthm} Suppose $\D$ is a reasonable class of $L^+$-expansions of the amalgamation class $(\K; \leq)$ of finite $L$-structures. Let $M$ denote the Fra\"{\i}ss\'e limit of $(\K; \leq)$. 
Then the space $X(\D)$ is a non-empty, compact space on which $\Aut(M)$ acts continuously.\hfill$\Box$
\end{theorem}

\begin{lemma} \label{reasonable}  Let $(\K; \leq)$ be an amalgamation class with Fra\"{\i}ss\'e limit $M$ and $G = \Aut(M)$. Let $\D$ be a reasonable class of expansions of $(\K; \leq)$ and suppose $Y \subseteq X(\D)$ is a subflow of the $G$-flow $X(\D)$. Then there is $\D_1 \subseteq \D$ which is a reasonable class of expansions of $(\K; \leq)$ such that $Y = X(\D_1)$. 
\end{lemma}

\begin{proof} Let $\D_1$ consist of isomorphism types of structures induced on $A$ by expansions in $Y$, for all finite $A \leq M$. Then $\D_1 \subseteq \D$ clearly satisfies properties (i), (ii) in the definition of reasonableness. Property (iii) follows from the $\leq$-homogeneity of $M$.

Clearly we have $Y \subseteq X(\D_1)$. We claim that $X(\D_1) \subseteq Y$. Let $N$ be an expansion of $M$ in  $X(\D_1)$. We show that $N$ is in the closure in $X(\D)$ of $Y$ and this will be enough. Suppose $A \leq M$ is finite. There is $N_1 \in Y$ and $B \leq M$ and an isomorphism from  $B^{N_1}$ (the induced structure on $B$ in $N_1$)  to $A^N$. By $\leq$-homogeneity of $M$, there is $g \in G$ which extends this map. By considering $N_1^g \in Y$ we obtain $N_2 \in Y$ with $A^{N_2} = A^{N_1}$. This gives what we need. 
\end{proof}

The following (from \cite{NVT}, Theorem 4; see also Proposition 5.5 in \cite{Z}) gives a criterion for minimality of the $\Aut(M)$-flow $X(\D)$. It relates to the notion of Expansion Property defined as follows:
\begin{definition}\label{EP:def}
Let $\D$ be a reasonable class of expansions of the  amalgamation class $(\K;\leq)$. We say that $\D$ has the \emph{Expansion Property} (\emph{EP} for short, or \emph{Lift Property} in~\cite{HN}) with respect to $(\K;\leq)$
if, for every $A\in \K$ there is $B \geq A$ in $\K$ with the property that for any expansions $A^+, B^+$ of $A, B$ in $\D$, there is an embedding $f : A^+ \to B^+$ which is $\leq$-strong. (If the extra structure imposed by $L^+$ is a total order, this is usually called the \emph{Ordering Property} (cf.~\cite{N}).)
\end{definition}

\begin{theorem} \label{EPthm} With the above notation, suppose that $\D$ is a reasonable class of expansions of the amalgamation class $(\K; \leq)$ and $M$ is the Fra\"{\i}ss\'e limit of $(\K; \leq)$. Then the $\Aut(M)$-flow $X(\D)$ is minimal if and only if $\D$ has the {Expansion Property} with respect to $(\K;\leq)$.\hfill$\Box$
\end{theorem}

\begin{remarks} Note that as every $G$-flow has a minimal subflow, it follows from the above that if $\D$ is a reasonable class of expansions of the amalgamation class $(\K; \leq)$, then there is a reasonable sub-class $\D_1 \subseteq \D$ which has the Expansion Property with respect to $(\K; \leq)$.
\end{remarks}

\medskip

Suppose, as in Theorem~\ref{sexpthm}, that $L \subseteq L^+$ are first-order languages with $L^+\setminus L$ relational. Suppose $(\K^+; \leq^+)$ is an amalgamation class of finite $L^+$-structures which is a strong expansion of the strong class $(\K; \leq)$ (cf. Definition~\ref{stex:def}). The latter is also an 
amalgamation class, by Theorem~\ref{sexpthm} and its Fra\"{\i}ss\'e limit $M$ is the $L$-reduct of the Fra\"{\i}ss\'e limit $N$ of $\ordpair{\K^+; \leq^+}$. Thus $\Aut(N)$ is a closed subgroup of $\Aut(M)$. We will say that $(\K^+; \leq^+)$ is a \emph{precompact} strong expansion of $(\K; \leq)$ when $\Aut(N)$ is a co-precompact subgroup of $\Aut(M)$. This means that every $\Aut(M)$-orbit on $M^n$ (for $n \in \N$) splits into finitely many $\Aut(N)$-orbits. This is a stronger condition that property (ii) in Definition~\ref{def:reasonable}. Thus, if  $(\K^+; \leq^+)$ is a precompact, strong expansion of $(\K; \leq)$, then $\K^+$ is a reasonable expansion of $(\K; \leq)$ if  property (iii) in Definition~\ref{def:reasonable} holds,  and in this case we can consider the $\Aut(M)$-flow $X(\K^+)$. 

%
%
%
\medskip

We then have the following version of \cite{KPT}, Theorem 10.8 and \cite{NVT}, Theorem 5. See also \cite{Z}, Theorem 5.7. 

\begin{theorem}\label{thm:comeagre} Let $L \subseteq L^+$ be first-order languages with $L^+\setminus L$ relational. Suppose $(\K; \leq)$ is a strong amalgamation class of finite $L$-structures with Fra\"{\i}ss\'e limit $M$. Suppose that $(\K^+; \leq^+)$ is a precompact, reasonable, strong expansion of $(\K; \leq)$ consisting of rigid $L^+$-structures. If $(\K^+; \leq^+)$ is a Ramsey class and $\K^+$ has the Expansion Property with respect to $(\K; \leq)$, then the $\Aut(M)$-flow $X(\K^+)$ is the universal minimal flow for $\Aut(M)$. It has a comeagre orbit consisting of expansions of $M$ isomorphic to the Fra\"{\i}ss\'e limit of $(\K^+; \leq^+)$.
\end{theorem}

We note a group-theoretic consequence of the above.
Suppose $G$ is a closed permutation group on a countable set.  We can regard $G$ as the automorphism group of some homogeneous structure $M$. Suppose $H$ is a closed, extremely amenable subgroup of $G$. We can consider this as the automorphism group of a homogeneous expansion $M^*$ of $M$. If this is a precompact expansion and $\Age(M^*)$ has the Expansion Property with respect to $\Age(M)$ (abusing terminology, we will say that $H$ has EP as a subgroup of $G$), then the above result  gives a description of $M(G)$, the universal minimal flow of $G$. Consider the quotient space $G/H$ with the quotient of the right uniformity on $G$ and denote by $\widehat{G/H}$ the completion. By precompactness, this is compact, metrizable and embeds $G/H$ homeomorphically as a comeagre set. It is a $G$-flow which, by the EP, is minimal. Extreme amenability of $H$ then implies that $\widehat{G/H}$ is isomorphic to $M(G)$ as a $G$-flow. 
This then yields the following result.

\begin{theorem}\label{X27} Suppose $G$ is a closed permutation group on a countable set and $H_1, H_2 \leq G$ are closed, extremely amenable, co-precompact subgroups of $G$ with EP. Then $H_1, H_2$ are conjugate in $G$. Moreover, $H_1$ is maximal amongst extremely amenable subgroups of $G$.
\end{theorem}

\begin{proof} The universal minimal flow $M(G)$ is isomorphic to  $\widehat{G/H_i}$ and $G/H_i$ is a comeagre $G$-orbit in this completion. 

So there is homeomorphism $\widehat{G/H_1} \to \widehat{G/H_2}$ which is a $G$-morphism. This must map the comeagre orbit to the comeagre orbit, so maps the coset $H_1$ to some coset $gH_2$. The stabilisers of these points must be identical, so $H_1 = gH_2g^{-1}$, as required. 

For the maximality part, we first observe that if $H_1 \leq H_1^g$, then $H_1^g = H_1$. By precompactness each $G$-orbit on $n$-tuples splits into a finite number of $H_1$-orbits, and the same number of $H_1^g$-orbits. It follows that $H_1, H_1^g$ have the same orbits on $n$-tuples and so are equal. Now suppose that $H_1 \leq H \leq G$ and $H$ is closed and extremely amenable. Clearly $H$ is co-precompact in $G$ and the $G$-map $gH_1 \mapsto gH$ extends to a continuous surjection $\widehat{G/H_1} \to \widehat{G/H}$. It follows that $\widehat{G/H}$ is a minimal $G$-flow and so $H$ has EP. From the above, we obtain $H=H_1$ as required.
\end{proof}

\begin{remarks}  By Theorem 6 of \cite{NVT}, we know that if a closed permutation group $G$ on a countable set has a closed, co-precompact, extremely amenable subgroup, then it has one, $H$, satisfying EP. By the above, such a subgroup $H$ is maximal amongst extremely amenable subgroups. The following example, pointed out to us by the Referee, shows that there can also be maximal, co-precompact subgroups which are not conjugate to $H$. Let $G$ be the full symmetric group on a countable set $M$ and $H$ the stabilizer of a dense linear order on $M$. So $H$ is co-precompact, extremely amenable and has EP. Let $a \in M$ and let $K$ be the stabilizer in $G_a$ of a dense linear order on $M\setminus \{a\}$. Then $K$ is again extremely amenable and co-precompact, but it is not contained in a conjugate of $H$ as $K$ fixes a point and has two orbits on $M$, whereas the stabilizer in $H$ of a point has 3 orbits on $M$.  \end{remarks}

\subsection{Comeagre orbits and the weak amalgamation property}

Suppose that $(\K; \leq)$ is an amalgamation class of $L$-structures and $\D$ is a reasonable class of expansions of $(\K; \leq)$. Then $(\D; \leq)$ is still a strong class, but of course it need not be an amalgamation class. 
Following \cite{KR} we say that $(\D; \leq)$ has the \emph{weak amalgamation property} if for all $A \in \D$, there is $B \in \D$ and a strong map $f : A \to B$ such that for all strong maps $f_i : B \to C_i \in \D$ (for $i=1,2$), there exist $D \in \D$ and strong maps $g_i : C_i \to D$ with $g_1(f_1(a)) = g_2(f_2(a))$ for all $a \in A$.  A similar property (the \emph{almost amalgamation property}) is introduced by Ivanov in \cite{Ivanov2}.  We then have:

\begin{lemma} \label{waplemma} Suppose that $\D$ is a reasonable class of expansions of the amalgamation class $(\K; \leq)$. Let $M$ be the \Fraisse{} limit of $(\K; \leq)$, let $G = \Aut(M)$ and consider the $G$-flow $X(\D)$. If $(\D; \leq)$ does not have the weak amalgamation property, then all $G$-orbits on $X(\D)$ are meagre. 
\end{lemma}

\begin{proof} Suppose that $A \in \D$ witnesses that $(\D; \leq)$ does not have the weak amalgamation property. Let $t \in X(\D)$ (so we think of this as the `extra structure' on $M$ for a particular expansion in $X(\D)$) and let $H$ denote the pointwise stabiliser in $G$ of $A$. We claim that the $H$-orbit $H\cdot t$ containing $t$ is nowhere-dense in $X(\D)$. As $H$ is of countable index in $G$, it then follows that the $G$-orbit $G\cdot t$ is a meagre subset of $X(\D)$. 

So suppose for a contradiction that $H\cdot t$ is dense in the open set $O$. We may assume that there is a finite $B \leq M$ with $A \leq B$ and $(B, s_B) \in \D$ such that $O = \{s \in X(\D) : s\vert B = s_B\}$. (Again, by the notation $(B, s_B)$ we mean the structure $B$ together with the additional structure it has as a member of $\D$.) Clearly we have that $t\vert A$, the induced structure on $A$ in $t$, is equal to  $s_B\vert A$. 

Note that, by reasonableness of $\D$, if $(B, s_B) \leq (C, s_C) \in \D$, then there is $s \in O$ and a $\leq$-embedding $f : (C, s_C) \to (M, s)$ which is the identity on $B$. As $H \cdot t$ is dense in $O$, there is $h \in H$ such that $h\cdot t \supseteq s \vert C$. Then $h^{-1}\circ f : (C, s_C) \to (M, t)$ is a $\leq$-embedding which is the identity on $A$. It follows that $(\D; \leq)$ has the weak amalgamation property over $A$: a contradiction.
\end{proof}

\begin{remarks} \rm Arguments in \cite{KR, Ivanov2} show that, with the above notation, if  $(\D; \leq)$ has the joint embedding property and the weak amalgamation property, then there is a comeagre $\Aut(M)$-orbit on $X(\D)$.
\end{remarks}

\subsection{EPPA and amenability} \label{sec:EPPA} The following is a modification of a well-known definition.

\begin{definition} \label{def:EPPA} Suppose $(\K; \leq)$ is a strong class of finite structures (as in Section~\ref{FCsec}). A \emph{strong partial automorphism} of $A \in \K$ is an isomorphism $f : D \to E$ for some $D, E \leq A$. We say that $(\K; \leq)$  has the \emph{extension property for strong partial automorphisms} (sometimes called the Hrushovski extension property or EPPA) if whenever $A \in \K$ there is $B \in \K$ with $A \leq B$ and such that every strong partial automorphism of $A$ extends to an automorphism of $B$. 
\end{definition}

 Recall that a topological group $G$ is \emph{amenable} if, whenever $Y$ is a $G$-flow, then there is a Borel probability measure $\mu$ on $Y$ which is invariant under the action of $G$. Thus, if $C \subseteq Y$ is a Borel set and $g \in G$, then $\mu(C) = \mu(gC)$. Of course, if $G$ is extremely amenable, then it is amenable (if $y \in Y$ is fixed by $G$, then take for $\mu$ the probability measure which concentrates on $y$). 

The following is  due to Kechris and Rosendal (\cite{KR}, Proposition 6.4). The terminology is as in Section~\ref{FCsec} here.

\begin{theorem}\label{thm:KR} Suppose $(\K; \leq)$ is an  amalgamation class of finite structures with Fra\"{\i}ss\'e limit $M$. Let $G = \Aut(M)$. Suppose $(\K; \leq)$ has the extension property for strong partial automorphisms (EPPA). Then: 
\begin{enumerate}
\item[(i)] There exist compact subgroups $G_1 \subseteq G_2 \subseteq \cdots$ of $G$ such that $\bigcup_i G_i$ is dense in $G$.
\item[(ii)] $G$ is amenable.

\end{enumerate}
\end{theorem}

Note that (ii) follows from (i) by standard results on amenability.

\section{$k$-Sparse graphs and their orientations}\label{sparsesec}

\subsection{The space of orientations}

In this section, the structures we work with are graphs and directed graphs (digraphs), so we use notation which is closer to standard graph-theoretic notation. Note that our directed graphs will be asymmetric: we do not allow loops nor vertices $a, b$ where both $a \to b$ and $b \to a$ are directed edges.

An undirected graph will be denoted as
$\Gamma = (A; R)$, $R \subseteq [A]^2$; if $B\subseteq A$ then $R^B = [B]^2\cap R$; so $(B; R^B)$ is the
 induced subgraph of $\Gamma$ on $B$. 

A digraph will be denoted as $\Delta = (A; S)$, where $S \subseteq A^2$.

\begin{definition}\label{vs}  Let $k \in \N$. We say that a graph $\Gamma = (A; R)$ is \emph{$k$-sparse} if for all finite $B \subseteq A$ we have $\lvert R^B \rvert \leq k \lvert B \rvert$. An infinite graph is \emph{sparse} if it is $k$-sparse for some $k \in \N$. 
\end{definition}
Note that this differs from the use of ``sparse''  in, for example, \cite{NP}.

\begin{remark} \label{arityn} We could consider the more general notion of a  sparse relational structure. For example, if $n \geq 2$ and $R_1 \subseteq A^n$ is an $n$-ary relation on $A$, then we say that $(A; R_1)$ is $k$-sparse if for all $B \subseteq A$ we have $\lvert R_1 \cap B^n\rvert \leq k\lvert B\rvert $. However, we can then consider the graph $(A; R)$ which has edges $\{a,b\}$ where $a,b \in \{s_1,\ldots, s_n\}$ for some $(s_1,\ldots, s_n) \in R_1$. This graph is $\binom{kn}{2}$-sparse. Thus, the results below apply more generally  to sparse relations.
Of course, dealing with graphs simplifies the reasoning.
\end{remark}

\begin{definition}  Let $k \in \N$. A graph $\Gamma = (A; R)$ is \emph{$k$-orientable} if there is a digraph $\Delta = (A; S)$ in which the out-valency of each vertex of $\Delta$ is at most $k$ and such that for all $a_1, a_2 \in A$, 
\[ \{a_1,a_2\} \in R \Leftrightarrow (a_1,a_2) \in S  \mbox{ or } (a_2, a_1) \in S \mbox{ (but not both)}.\]
In this case, we refer to $\Delta$ (or $S$) as a \emph{$k$-orientation} of $\Gamma$.
\end{definition}

So, informally, a $k$-orientation of $\Gamma$ is obtained by choosing a direction on each edge of $\Gamma$ in such a way that no vertex has more than $k$ directed edges coming out of it. Note that if a graph $\Gamma = (A; R)$ is $k$-orientable, then its edge-set can be decomposed into subsets $R_1,\ldots, R_k$ such that each graph $(A; R_i)$ is $1$-orientable. Moreover, a graph is $1$-orientable if and only if each of its connected components is a `near-tree': a tree with at most one extra edge. Thus, $k$-orientability is closely related to \emph{$k$-arboricity}. The following is well-known to graph-theorists~\cite{NW}, but we include a proof.

\begin{theorem} \label{23}A graph $\Gamma = (A; R)$ is $k$-orientable if and only if it is $k$-sparse.
\end{theorem}

\begin{proof} If $\Gamma = (A; R)$ is $k$-orientable and $B$ is a finite subset of $A$, then then the number of edges in $R^B$ is (at most) the number of directed edges in the induced sub-digraph on $B$ in a $k$-orientation of $\Gamma$. Clearly this is at most $k\lvert B\rvert $.

Note that by a compactness (or K\"onig's Lemma) argument, it suffices to prove the converse in the case where $A$ is finite, so we now assume this. For a $k$-orientation of $\Gamma$ we need to choose, for each edge $e = \{a,b\} \in R$ one of the vertices $a,b$ to be the initial vertex of the directed edge. We need to do this so that the resulting digraph has out-valency at most $k$. 

Consider $k$ copies $A \times [k]$ of the vertex set $A$ (where $[k] = \{1,\ldots, k\}$) and form a bipartite graph $B$ with parts $R$ and $A \times [k]$. We have an edge $(e, (a,l))$ (where $e \in R$ and $a \in A$, $l \leq k$) in this bipartite graph if and only if $a \in e$. We show that the condition of Hall's Marriage Theorem holds and hence there is a matching of $R$ into $A \times [k]$. Indeed, if $I \subseteq R$, let $C \subseteq A$ be the union of the edges in $I$. Then the number of vertices in $A\times[k]$ adjacent to $I$ is $k\lvert C\rvert $ and 
$ k\lvert C\rvert \geq \lvert R^C\rvert  \geq \lvert I\rvert ,$
as required. 

Fix a matching in $B$. We orient an edge $e = \{a,b\}$ of $\Gamma$ by taking the directed edge $(a,b)$ precisely when $e$ is matched with some $(a,l)$ under the matching. This is a $k$-orientation.
\end{proof}

We use the following special case of the construction in Theorem~\ref{XDthm}.

\begin{definition} \label{or}  Suppose that $\Gamma = (A; R)$ is a $k$-sparse graph. We let 
\[ X_{\Gamma} = \{ S \subseteq A^2 : (A; S) \mbox{ is a $k$-orientation of } \Gamma\}\]
be the set of $k$-orientations of $\Gamma$. Identifying $S \in X_{\Gamma}$ with its characteristic function, we can view $X_{\Gamma}$ as a subset of $\{0,1\}^{A^2}$. We give the latter the product topology (where $\{0,1\}$ has the discrete topology). We give $X_{\Gamma}$ the subspace topology and refer to it as the \emph{space of $k$-orientations} of $\Gamma$. Note that the automorphism group $\Aut(\Gamma)$ acts continuously on $\{0,1\}^{A^2}$ via its diagonal action on $A^2$ and $X_{\Gamma}$ is invariant under this action. 
\end{definition}

Of course, this depends on the particular $k$, but we omit this dependence in the notation.

\begin{lemma} Suppose $\Gamma = (A; R)$ is a $k$-sparse graph. Then $X_{\Gamma}$ is an $Aut(\Gamma)$-flow.
\end{lemma}

\begin{proof} By Theorem~\ref{23}, $X_{\Gamma}$ is non-empty. We know that $X_{\Gamma}$ is an invariant subspace of the $\Aut(\Gamma)$-flow $Y = \{0,1\}^{A^2}$, so it suffices to observe that it is a closed subspace. But if $S \in Y \setminus X_{\Gamma}$ is not a $k$-orientation of $\Gamma$, then this is witnessed on some finite subset $C$ of $A$. So if $S' \in Y$ agrees with $S$ on $C^2$, then $S' \not\in X_{\Gamma}$. Thus $Y \setminus X_{\Gamma}$ is open and $X_{\Gamma}$ is closed.
\end{proof}

\subsection{Extremely amenable subgroups}

 Suppose $G$ is a topological group acting by automorphisms on a discrete structure $M$. The action $G \times M \to M$ is continuous if and only if stabilizers in $G$  of points of $M$ are open in $G$. Equivalently, the induced homomorphism $G \to \Aut(M)$ is continuous. In this case, we say that $G$ is acting \emph{continuously (by automorphisms)} on $M$, often omitting the phrase `by automorphisms'. Note that in this case, the induced action of $G$ on a space such as $\{0,1\}^{M^2}$ is also continuous.

The following is the first main, new result of the paper and leads quickly to Theorem~\ref{cex}.

\begin{theorem} \label{sparsethm}Suppose $k \in \N$ and $\Gamma = (M; R)$ is a $k$-sparse graph in which all vertices have infinite valency. Suppose $G$ is a topological group which acts continuously (by automorphisms) on $\Gamma$. If $H \leq G$ is extremely amenable, then $H$ has infinitely many orbits on $M^2$.
\end{theorem}

\begin{proof}  Consider $G$ acting on the space $X_{\Gamma}$ of $k$-orientations of $\Gamma$. This is a $G$-flow and so, as $H$ is extremely amenable, there is some $S \in X_{\Gamma}$ which is fixed by $H$.  So $H$ is acting as a group of automorphisms of the digraph $(M; S)$.

To finish the proof, it will suffice to show that $K = \Aut(M; S)$ has infinitely many orbits on $M^2$. Suppose not. Then $K$ has finitely many orbits on $M$ and for every $a \in M$ the pointwise stabiliser $K_a$ of $a$ in $K$ has finitely many orbits on $M$. Furthermore, if $b \in M$ and  there is a directed path of length $r$ from $a$ to $b$ in $(M; S)$, then $b$ lies in an orbit of size at most $k^r$ under $K_a$. It follows that there is a  bound $l$, independent of $a$,  on the size of the set of vertices reachable by a directed path starting from $a$. Take the smallest such $l$ and suppose $a$ realises this: so the set $A$ of vertices reachable by a directed path from $a$ (including $a$) is of size $l$.  As $A$ is finite and $a$ has infinite valency in $(M; R)$, there is a vertex $c \not\in A$ which is adjacent to $a$. In $(M; S)$ this edge must be directed from $c$ to $a$. So the set of vertices which are reachable by a directed path starting at $c$ has size at least $l+1$: contradiction. 
\end{proof}

\begin{proof}[Proof of Theorem~\ref{cex}] We note that the variation on the Hrushovski construction in Section~\ref{Homegacat} produces a countable, sparse, $\omega$-categorical graph. It is easy to see that all vertices in this graph have infinite valency, so Theorem~\ref{sparsethm} gives the required result.
\end{proof}

\subsection{Non-amenability}\label{nonambsec} 
In the following, if $G$ is a group acting on a set $M$ and $a \in M$, then $G_a = \{g \in G : ga=a\}$, the stabilizer of $a$ in $G$.

\begin{theorem} \label{38} Suppose $k \in \N$ and $\Gamma = (M; R)$ is a $k$-sparse graph and $G$ is a topological group which acts continuously (by automorphisms) on $\Gamma$.  Suppose there are adjacent vertices $a, b$ in $\Gamma$ such that the $G_a$-orbit containing $b$ and the $G_b$-orbit containing $a$ are both infinite. Then $G$ is not amenable.
\end{theorem}

\begin{proof} The following is based on an argument of Todor Tsankov and replaces our original, less general argument.

We show that there is no $G$-invariant Borel probability measure on the $G$-flow  $X_{\Gamma}$. Suppose, for a contradiction, that $\mu$ is such a measure. 

Let $a,b$ be as in the statement and consider the open set $S_{ab} = \{ S \in X_{\Gamma} : (a,b) \in S\}$, the orientations in which this edge is directed from $a$ to $b$. As $S_{ab}\cup S_{ba} = X_{\Gamma}$ we may assume that $\mu(S_{ab}) = p \neq 0$. For $r \in \N$, let $b_1,\ldots, b_r$ be distinct elements of the $G_a$-orbit containing $b$. So $\mu(S_{ab_i}) = p$ for each $i \leq r$. 

Let $s_i$ be the characteristic function of $S_{ab_i}$. Then for every $k$-orientation $S \in X_{\Gamma}$ we have 
\[\sum_{i \leq r} s_i(S) \leq k.\]
Thus 
\[\int_{S \in X_{\Gamma}} \sum_{i \leq r} s_i(S)\, d\mu(S) \leq k.\]
On the other hand we have $$\int_{S\in X_{\Gamma}} s_i(S) d\mu(S) = p,$$ therefore $rp \leq k$. As $p \neq 0$ and $r$ is unbounded, this is a contradiction.
\end{proof}

\begin{corollary} \label{noEPPA} Suppose $k \in \N$ and $(\K; \leq)$ is an amalgamation class of $k$-sparse graphs (possibly carrying extra structure). Let $M$ denote the Fra\"{\i}ss\'e limit of $(\K; \leq)$ and $G = \Aut(M)$. Suppose there are adjacent vertices $a, b$ in $M$ such that the $G_a$-orbit of $b$ and the $G_b$-orbit of $a$ are both infinite. Then $(\K; \leq)$ does not have EPPA.
\end{corollary}

\begin{proof} Note that the graph on $M$ is sparse, so by Theorem~\ref{38}, $G$ is not amenable. It then follows from  Theorem~\ref{thm:KR} that $(\K; \leq)$ does not have EPPA.
\end{proof}

\begin{remark} We could rephrase the assumption on $G$ in the above as a condition on $(\K; \leq)$ and in general, it is a straightforward matter to check this. We will illustrate this below where $(\K; \leq)$ will be a Hrushovski amalgamation class of sparse graphs.
\end{remark}

We note that the arguments in Theorem~\ref{sparsethm} and \ref{38} can be combined to show the following strengthening of Theorem~\ref{cex}.

\begin{corollary} \label{noamenable} Suppose $k \in \N$ and $\Gamma = (M; R)$ is a $k$-sparse graph in which all vertices have infinite valency. Suppose $G$ is a topological group which acts continuously (by automorphisms) on $\Gamma$. If $H \leq G$ is  amenable, then $H$ has infinitely many orbits on $M^2$. In particular, there is a countable  $\omega$-categorical structure $M$ such that $\Aut(M)$ has no co-precompact amenable subgroup. \hfill $\Box$
\end{corollary}

\subsection{Key Examples} \label{sec:C0} To illustrate further the method used in the above results, we describe the amalgamation class $(\C_0; \leq_s)$ of $k$-sparse graphs and the associated class $(\D_0; \sleq)$ of $k$-orientations ($k$ is fixed and understood from the context). In the next section, we will see (following \cite{DEtrivial}) that $(\C_0; \leq_s)$ can also be understood as a special case of Hrushovski's predimension construction. The Fra\"{\i}ss\'e limit $M_0$ of $(\C_0; \leq_s)$ is not $\omega$-categorical. However, by using a version of the predimension construction, we can define a strong class $(\C_F; \leq_d)$ of $k$-sparse graphs where the Fra\"{\i}ss\'e limit is $\omega$-categorical.

Fix an integer $k \geq 2$. Formally, we can consider structures in $\C_0$ as $L$-structures where $L$ is a language with a binary relation symbol $R$ for the edges. We consider structures in $\D_0$ as structures in the expanded language $L^+$ which also has a binary relation symbol $S$ for the directed edges (and $R$ still records the undirected edges). 

\begin{definition}\label{def:D0} Let $\D_0$ consist of the finite $k$-oriented digraphs, that is, directed graphs where the out-degree of every vertex is at most $k$.  If $A \in \D_0$ and $B \subseteq A$, we write $B \sleq A$ to mean that if $b \in B$ and $b \to a$ is a directed edge in $A$, then $a \in B$ (so $B$ is closed under successors in $A$). If $C \subseteq A \in \D_0$, 
 we write $\scl_A(C)$ for the successor-closure of $C$ in $A$. So $\scl_A(C)$ is the smallest subset $B$ of $A$ containing $C$ with $B \sleq A$. 
 \end{definition}
 
We let $\C_0$ be the class of $k$-sparse graphs. By Theorem~\ref{23}, this is the class of undirected reducts of $\D_0$. 

\begin{definition} For $A \subseteq B \in \C_0$ we write $A \leq_s B$ if there is a $k$-orientation $B^+$ of $B$ in which $A \sleq B^+$. 
\end{definition}

It is easy to see that $(\C_0; \leq_s)$ is a strong class: essentially we need to verify that if $A \leq_s B \leq_s C \in \C_0$ then $A \leq_s C$. To see this, note that there is a $k$-orientation $C^+$ of $C$ in which $B$ is a successor-closed subset. If we replace the induced orientation on $B$ by any other $k$-orientation of $B$, the result is still a $k$-orientation of $C$. So we choose an orientation of $B$ in which $A$ is successor-closed and obtain an orientation of $C$ in which both $A$ and $B$ are successor-closed.

In the terminology of Theorem~\ref{sexpthm}, the same argument shows:

\begin{theorem} The class $(\D_0; \sleq)$ is a strong expansion of the class $(\C_0; \leq_s)$. Both of these are free amalgamation classes.
\end{theorem}

\begin{proof} It is clear that $(\D_0; \sleq)$ is a free amalgamation class. It then follows from Theorem~\ref{sexpthm} that $(\C_0; \leq_s)$ is a free amalgamation class. 
\end{proof}

\begin{definition} Let $M_0$ and $N_0$ denote respectively  the Fra\"{\i}ss\'e limits of $(\C_0; \leq_s)$ and $(\D_0; \sleq)$. 
\end{definition}

Note that by Theorem~\ref{sexpthm}, $M_0$ is the undirected reduct of the $k$-oriented digraph $N_0$, so $\Aut(N_0)$ is a subgroup of $\Aut(M_0)$. However, it is not a co-precompact subgroup. Each $A \in \C_0$ has only finitely many expansions in $\D_0$, but there is no bound on the size of the successor-closure of $f(A)$ in $N_0$ for $\leq_s$-strong embeddings $f : A \to N_0$. We also note that $\D_0$ is a reasonable class of expansions of $(\C_0; \leq_s)$.
 
 \medskip

%
%
%
%
%
%

\begin{theorem}\label{CorC0} \leavevmode \begin{enumerate} \item The group $\Aut(M_0)$  has no co-precompact extremely amenable subgroup. Equivalently, there is no precompact Ramsey expansion of the strong class $\ordpair{\C_0; \leq_s}$. 
\item The group  $\Aut(M_0)$ is not amenable and the strong class $\ordpair{\C_0; \leq_s}$ does not have EPPA.
\end{enumerate}
\end{theorem}

\begin{proof}  To reduce notation, let $M$ denote $M_0$ and $G = \Aut(M)$. Let $\ordpair{\C; \leq}$ denote $(\C_0; \leq_s)$.

(1) Let $A \in \C$ consist of two non-adjacent vertices and $P$ be the set of $\leq$-copies of $A$ in $M$. So this is an $\Aut(M)$-orbit on $M^2$. Consider the $G$-flow of $k$-orientations of $M$. As every extremely amenable subgroup of $\Aut(M)$ must fix some element of this, it is enough to show that if $N$ is a $k$-orientation of $M$, then $\Aut(N)$ has infinitely many orbits on $P$. Suppose, for a contradiction, that there are only a finite number $t$ of $\Aut(N)$-orbits on $P$.

Suppose $m \in \N$ is arbitrary. By considering a tree of height $m$ and non-leaf vertices having valency $k+1$ as a $\leq$-substructure of $M$, there is a directed path $Q_m \leq N$ of length $m+1$. Label the vertices as $a=a_0, a_1,\ldots, a_m$ (with $a_{i} \to a_{i+1}$ a directed edge in $N$). If $i \geq 1$ then  $(a,a_i)\leq Q_m$,  so $(a,a_i) \in P$.

We show that if $m$ is large enough in relation to $t$, then the $a_i$ lie in more than $t$ different orbits under the stabiliser $K$ of $a$ in $\Aut(N)$. This is a contradiction.

Let $B(a; i)$ denote the set of vertices $b$ in $N$ for which there is a directed path of length at most $i$ from $a$ to $b$. Thus $a_0,\ldots, a_i \in B(a; i)$ and $\vert B(a; i)\vert \leq k^{i+1}-1$. Let $s(1) = 1$ and $s(n+1) = k^{s(n)+1}$, for $n \in \N$. We show that if $m \geq s(t)$ then the points $a_1,\ldots, a_m$ lie in at least different $K$-orbits. We prove this by induction on $t$, the case $t = 1$ being trivial. If the result holds for $t$ but not for $t+1$ then each of $a_{s(t)+1},\ldots, a_{s(t+1)}$ is in the same orbit as one of the points $a_1,\ldots, a_{s(t)}$. Thus $\{a_1,\ldots, a_{s(t+1)} \} \subseteq B(a; s(t))$. As $s(t+1) > \vert B(a; s(t))\vert$, this is a contradiction.

(2) Let $\{a,b\} \leq M$ with $a$ adjacent to $b$. The $G_a$-orbit containing $b$ and the $G_b$-orbit containing $a$ are both infinite. To see this, note that the graph $B_n$ with vertices $a, b_1,\ldots, b_n$ and edges $ab_i$ is in the class $\C$ and $\{a\} \leq \{a, b_i\} \leq B_n$ (where $n$ is arbitrary). So we can regard $B_n$ as a subgraph of $M$. By the $\leq$-homogeneity of $M$, the vertices $b_1,\ldots, b_n$ are in the same $G_a$-orbit as $b$. As $n$ is arbitrary here, the $G_a$-orbit of $b$ is infinite. Similarly, the $G_b$-orbit of $a$ is infinite. The result now follows from Theorem~\ref{38} and Corollary~\ref{noEPPA}.\end{proof}

\begin{remarks} (1) By the above result and Theorem 1.2 of \cite{Z}, we know that some structure in $(\C_0; \leq_s)$ has infinite Ramsey degree (as defined in Section~\ref{rcsec}). In fact, the argument in the proof of part (1) above shows that, in the notation used there,  $A$ has infinite Ramsey degree in $\ordpair{\C_0; \leq_s}$. A longer argument can be used to show that the structure consisting of a single vertex also has infinite Ramsey degree in $\ordpair{\C_0; \leq_s}$.

(2) It would be interesting to know whether the argument can be extended to show that $\Aut(M_0)$ has no amenable co-precompact subgroups, as in Corollary~\ref{noamenable}.
\end{remarks}

\section{Hrushovski's predimension construction}
\label{hcsec}

In this section we give a short account of Hrushovski's predimension construction of an $\omega$-categorical sparse graph from \cite{Hrpp, Hr}. We will make use of and extend the connection with orientations as in \cite{DEtrivial}. More traditional approaches to the construction can be found in  \cite{DEBonn}, or Wagner's article \cite{Wa}. 

\subsection{Predimension and roots}

Let $k \geq 2$ be a fixed integer and let $(\C_0; \leq_s)$ be the class of all finite $k$-sparse graphs as in Section~\ref{sec:C0}. If $A \in \C_0$, the \emph{predimension} of $A$ is 
\[\delta(A) = k \lvert A \rvert - \lvert R^A\rvert.\]

Note that for a graph in $\C_0$, this predimension is always non-negative. In fact, a graph is in $\C_0$ precisely when $\delta(A) \geq 0$ for all  finite subgraphs $A$. The following extension of Theorem~\ref{23}  is from \cite{DEtrivial}.

\begin{lemma}\label{Hall2} Suppose $A \subseteq B \in \C_0$. Then $A \leq_s B$ if and only if $\delta(A) \leq \delta(C)$ whenever $A \subseteq C \subseteq B$. \hfill$\Box$
\end{lemma}

The following gives another way of thinking about predimension.

\begin{definition} 
If  $A \in \D_0$ is a $k$-oriented digraph, we say that a vertex $v\in A$ is a \emph{root} in $A$ if its out-degree $r$ in $A$ is less than $k$ (of course, $k$ is understood from the context here). The
\emph{multiplicity} of $v$ is then $k-r$.
\end{definition} 

Note that if $B \in \D_0$ and $A \sleq B$, then a vertex in $A$ is a root in $A$ if and only if it is a root in $B$. 

 A simple counting argument gives:

\begin{lemma}\label{lem:roots}Given a $k$-oriented graph $A$, its predimension $\delta(A)$ is equal to the sum of 
the multiplicities of its roots in $A$.\hfill$\Box$
\end{lemma}

For Hrushovski's construction of a sparse $\omega$-categorical structure we will need the following definition. 

\begin{definition}  Suppose that $A \subseteq B \in \C_0$. We write $A \leq_d B$ (and say that $A$ is \emph{$d$-closed} in $B$) if $\delta(A) < \delta(C)$ whenever $A \subset C \subseteq B$.
\end{definition}

\begin{definition} \label{def:sdcl} Given a  $k$-orientation $A$ and  a subdigraph $B$ we denote by $\Roots_A(B)$ the set of elements of  the
successor-closure $\scl_A(B)$ of $B$ in $A$ which are roots in $A$. The \emph{successor-$d$-closure} of $B$ in $A$, denoted by $\sdcl_A(B)$,
is the set of all vertices $v\in A$ such that $\Roots_A(v)\subseteq \Roots_A(B)$. If $B$ is successor-$d$-closed
in $A$ we write $B\dleq A$.
\end{definition} 

The terminology is justified by the following lemma.

\begin{lemma}
\label{lem:sdcl}
Suppose $A\in \D_0$ and $B\subseteq A$. Then  $\sdcl_A(B)$ is the smallest
substructure of $A$ containing $B$ that is both successor-closed and $d$-closed.
\end{lemma}
\begin{proof}
It is easy to see that $E = \sdcl_A(B)$ is successor-closed in $A$; we show that it is $d$-closed in $A$.  Suppose  $E \subset C \subseteq A$.  Then $E \sleq C$ and so by Lemma~\ref{lem:roots}, we have $\delta(E) \leq \delta(C)$. If $\delta(E) = \delta(C)$, then no vertex  $x \in C\setminus E$ is a root in $C$, and therefore all such vertices have out-valency $k$ in $C$. It follows that $C \sleq A$. But then $\Roots_A(x) = \Roots_C(x) \subseteq E$ for all $x \in C$ and therefore $C = E$, a contradiction. Thus $E \leq_d A$, as required. 

Now suppose $B \subseteq F \sleq A$ and $F \leq_d A$. If $x \in \sdcl_A(B)$ then $\Roots_A(y) \subseteq F$ for all $y \in \scl_A(x)$. Thus, by Lemma~\ref{lem:roots}, $\delta(F \cup \scl_A(x)) = \delta(F)$ and so $x \in F$. 
\end{proof}

Note that for a digraph in $\D_0$, the property of a subset being $d$-closed is determined by the undirected reduct. 

We also have:

\begin{lemma} \label{48} Suppose $A \subseteq B \in \C_0$. Then $A \leq_d B$ if and only if there is a $k$-orientation $B^+ \in \D_0$ of $B$ where $A \dleq B^+$. 
\end{lemma}

\begin{proof} Suppose $A \dleq B$.  It is enough to show that $\delta(A) > \delta(B)$, assuming $A \neq B$.  Let $b \in B \setminus A$. Then  $\Roots_B(b) \not\subseteq A$, so $\delta(B) > \delta(A)$, by Lemma~\ref{lem:roots}.

Conversely, suppose $A \leq_d B$. By Lemma~\ref{Hall2}, there is a $k$-orientation $B^+$ of $B$ in which $A \sleq B^+$. We claim $A \dleq B^+$. Indeed, if $b \in B$ and $\Roots_B(b) \subseteq A$, then $\delta(\scl_B(b) \cup A) = \delta(A)$, so $b \in A$. 
\end{proof}

\begin{lemma} \label{310}Let $B \in {\C_0}$ and let $\leq$ denote either $\leq_s$ or $\leq_d$. 
\begin{enumerate}
\item If $A \leq B$  and $X \subseteq B$, then $A \cap X \leq X$.
\item If  $A \leq C \leq B$, then $A \leq B$.
\item If  $A_1, A_2 \leq B$, then $A_1 \cap A_2 \leq B$.
\end{enumerate}
\end{lemma}

\begin{proof} We verify these in the case of $\leq_d$. For (1), by Lemma~\ref{48}, there is an orientation $B^+$ of $B$ with $A \dleq B^+$. If $b \in X$ and $\Roots_X(b) \subseteq X \cap A$, then every vertex $c \in \scl_X(b) \setminus (X \cap A)$ has out-valency $k$ in $X$. Thus its successors in $B$ are in $X$. It follows that $\Roots_B(b) \subseteq A$, so $b \in A\cap X$, as required. Similarly, for (2), there is an orientation $B^+$ of $B$ with $A \dleq B^+$ and an orientation $C^+$ of $C$ in which $B \dleq C^+$. Replacing the structure on $B$ in $C^+$ by $B^+$, we obtain a $k$-orientation $C^{++}$ of $C$ in which $B^+ \dleq C^{++}$. But then $A \dleq C^{++}$. So $A \leq_d C$. (3) follows from (1) and (2).
\end{proof}

As in Section~\ref{sec:C0}, we then have:

\begin{lemma} The classes $(\D_0; \dleq)$ and $(\C_0; \leq_d)$ are strong classes. Moreover, $(\D_0; \dleq)$ is a strong expansion of $(\C_0; \leq_d)$. \hfill$\Box$
\end{lemma}

The fact that both $\leq_s$ and $\leq_d$ give rise to closures follows from  Lemma~\ref{310}. We shall be particularly concerned with $d$-closure.

\begin{definition} With this notation, if $B \in \C_0$ then Lemma~\ref{310} (3) shows that if $A \subseteq B$ and $S = \{A_1 : A \subseteq A_1 \leq_d B\}$, then $\bigcap S \leq_d B$. So there is a smallest $\leq_d$-subset of $B$ which contains $A$: denote it by $\cl_B^d(A)$. It is easy to see that $\cl^d_B$ is a closure operation on $B$.
\end{definition}

\begin{lemma}\label{312} For $A \subseteq B \in {\C}$ we have $\delta(A) \geq \delta(\cl^d_B(A))$.
\end{lemma}

\begin{remarks}  As we already mentioned, the original approach to this works with a more general predimension. Let $\alpha$ be a positive real number and for a graph $A = (A; R^A)$ let $\delta(A) = \alpha \lvert A \rvert - \lvert R^A\rvert$. Note that for $B, C \subseteq A$ we have $\delta(B\cup C) \leq \delta(B)+\delta(C) - \delta(B\cap C)$. Using this, one can prove all of the above properties of $\leq_s$ and $\leq_d$. 
\end{remarks}

\subsection{The $\omega$-categorical case}

\label{Homegacat}

Following \cite{Hrpp}, we will consider  subclasses of $(\C_0; \leq_d)$ in which $d$-closure is uniformly bounded.
 More precisely we use the following definition.

\begin{definition}\label{CF} Let $F : \R^{\geq 0} \to \R^{\geq 0}$ be a continuous, increasing function with $F(x) \to \infty$ as $x \to \infty$, and $F(0) = 0$. Let 
\[\C_F = \{ B \in \C_0 : \delta(A) \geq F(\lvert A \rvert) \mbox{ for all } A \subseteq B\}.\]
\end{definition}

\begin{theorem}\leavevmode \begin{enumerate}
\item If $B \in \C_F$ and $A \subseteq B$ then 
\[\lvert \cl_B^d(A)\rvert \leq F^{-1}(k\lvert A \rvert).\]
\item If $(\C_F; \leq_d)$ is an amalgamation class, then its \Fraisse{} limit $M_F$ is $\omega$-categorical.
\end{enumerate}
\end{theorem}

\begin{proof} (1) By Lemma~\ref{312} we have $\delta(\cl_B^d(A)) \leq \delta(A) \leq k\lvert A \rvert$. Thus (by definition of $\C_F$) we have $\lvert \cl_B^d(A)\rvert \leq F^{-1}(k\lvert A \rvert)$.

(2) This follows from Remarks~\ref{34}.
\end{proof}

We now provide some examples (taken from \cite{Hrpp}) where $(\C_F; \leq_d)$ is a free amalgamation class (in the sense of Remarks~\ref{freeamcl}).

\begin{example}\label{320} Let $F$ as in Definition~\ref{CF} be such that:
\begin{itemize}
\item $F$ is piecewise smooth;
\item the right derivative $F'$ is non-increasing;
\item $F'(x) \leq 1/x$ for all $x > 0$. 
\end{itemize}
Then we claim that $(\C_F; \leq_d)$ is a free amalgamation class.

Indeed, suppose $A \leq_d B_1, B_2 \in \C_F$ and let $E$ be the free amalgam of $B_1$ and $B_2$ over $A$. We need to show that $E \in \C_F$. Clearly we may assume $A \neq B_i$. 

Suppose $X \subseteq E$. We need to show that $\delta(X) \geq F(\lvert X \rvert)$. Now, $X$ is the free amalgam over $A \cap X$ of $B_1\cap X$ and $B_2 \cap X$ and $A \cap X \leq_d B_i\cap X$ (by Lemma~\ref{310}(1)). So we can assume $X = E$ and check that $\delta(E) \geq F(\lvert E \rvert)$.

Note that $\delta(E) = \delta(B_1)+\delta(B_2) - \delta(A)$ and $\lvert E\rvert = \lvert B_1\rvert + \lvert B_2\rvert - \lvert A \rvert$. 

The effect of the conditions on $F$ is that for $x, y \geq 0$
\[ F(x+y) \leq F(x) + yF'(x) \leq F(x)+ y/x.\]

We can assume that 
\[\frac{\delta(B_2) - \delta(A)}{\lvert B_2\rvert - \lvert A\rvert} \geq \frac{\delta(B_1) - \delta(A)}{\lvert B_1\rvert - \lvert A\rvert}\]
and note that the latter is at least $1/\lvert B_1 \rvert $ (as $\delta$ is integer-valued and $A\leq_d B_1$).

Then 

\begin{eqnarray*}
\delta(E) &=& \delta(B_1) + (\lvert B_2\rvert - \lvert A\rvert)\frac{\delta(B_2) - \delta(A)}{\lvert B_2\rvert - \lvert A\rvert} \\
&\geq& F(\lvert B_1\rvert) +(\lvert B_2\rvert - \lvert A\rvert)/\lvert B_1\rvert \\
&\geq& F(\lvert E \rvert)
\end{eqnarray*}
(taking $x = \lvert B_1\rvert$ and $y = \lvert B_2 \rvert - \lvert A \rvert$). 

\begin{figure}
\centering
\includegraphics{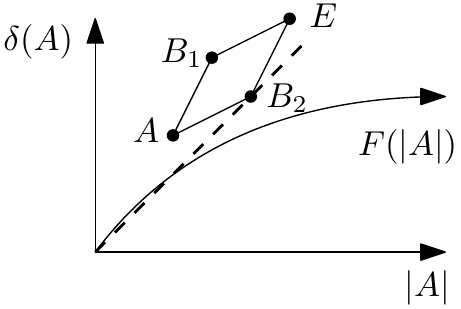}
\caption{Predimension of the free amalgamation of $B_1$ and $B_2$ over $A$.}
\label{fig:predimension}
\end{figure}
This concludes the proof of the claim (see Figure~\ref{fig:predimension}).
\end{example}

\begin{example} \label{321} In order to illustrate the flexibility of this, we use the construction to produce an example of a connected $\omega$-categorical graph whose automorphism group is transitive on vertices and edges, and whose smallest cycle is a $5$-gon.

Let $k = 2$. So we are working with $2$-sparse graphs and the predimension:
\[ \delta(A) = 2 \lvert A \rvert - \lvert R^A \rvert.\]

Take
$$F(1) = 2; F(2) = 3; F(5) = 5; F(k) = \log(k) + 5 - \log(5) \mbox{ for $k \geq 5$}.$$
Then one can check that:
\begin{itemize}
\item The smallest cycle in $\C_F$ is a 5-gon.
\item If $a \in A \in \C_F$ then $\{a\} \leq_d A$.
\item If $\{a,b\} \subseteq B \in \C_F$ is an edge then $\{a,b\} \leq_d B$.
\item $(\C_F; \leq_d)$ is an amalgamation class (the proof of the amalgamation property  in the previous example applies if at least one of $B_1, B_2$ has size $\geq 5$; the other cases can be checked individually).
\item The Fra\"{\i}ss\'e limit $M_F$ is connected. Given non-adjacent $a, b \in M_F$ consider $A = \cl^d(\{a,b\})$. As $\delta(A) \leq \delta(\{a,b\}) = 4$ we have $\lvert A \rvert \leq 3$. So either $A$ is a path of length 2 (with endpoints $a, b$) or $A = \{a,b\}$, so $\{a,b\} \leq_d M_F$. In the latter case, consider a path $B$ of length $3$ with end points $a,b$. Then $\{a,b\}\leq_d B$ so there is a $\leq_d$ copy of $B$ in $M_F$ over $\{a,b\}$. In particular, $a,b$ are at distance $3$ in $M_F$. 
\end{itemize}
\end{example}
Note that the Ramsey properties of classes of graphs with large girth are a difficult combinatorial problem (cf. \cite{N}).

\begin{remark}

Consider an arbitrary graph $G$ and a $2$-orientation $\widehat G$ created by subdividing every
edge $\{a,b\}$ of $G$ by a new vertex $v_{a,b}$ and orienting the subdivision as $v_{\{a,b\}}\to a$, $v_{\{a,b\}}\to b$.
It is easy to see that for $A\subseteq \widehat G$, $\delta(A)>\sqrt{\lvert A\rvert }$ and thus there are choices of $F$ such that for every finite graph $G$ it holds that $\widehat G\in \C_F$.
The class of $\ordpair{\D;\dleq}$ of such representations of graphs is a free amalgamation class and is bi-definable with
the class of all graphs.  (This is not the case with $\C_0$.) 

\end{remark}

\subsection{Results in the $\omega$-categorical case} 

We can now conclude the proof of: 

\medskip

\noindent\textbf{Theorem~\ref{cex}.}  \textit{\quad There exists a countable, $\omega$-categorical structure $M$ with the property that if $H \leq \Aut(M)$ is extremely amenable, then $H$ has infinitely many orbits on $M^2$. In particular, there is no $\omega$-categorical expansion of $M$ whose automorphism group is extremely amenable.}

\medskip

\begin{proof} Let $M_F$ be the $\omega$-categorical $2$-sparse graph constructed in Example~\ref{321}. All vertices of $M_F$ are of infinite valency (as in the proof of Theorem~\ref{CorC0}), so the result follows from Theorem~\ref{sparsethm}. 
\end{proof}

We also note that, as in the proof of Theorem~\ref{CorC0}: 

\begin{corollary} The group $\Aut(M_F)$ is not amenable and the class $(\C_F; \leq_d)$ does not have EPPA.
\end{corollary}

 All that we require of our free amalgamation class $\C_F$ is that it contains an `edge' $\{a,b\}$ (with $R(a,b)$ holding). In this case, $\{a\} , \{b\} \leq_d \{a,b\}$ and the argument of Theorem~\ref{CorC0} applies. 
 
 \begin{remarks}   in \cite{Gh}, Zaniar Ghadernezhad gives a direct argument to show that the class $(\C_0; \leq_s)$ fails to have EPPA. It would be interesting to have a similar direct proof for the failure of EPPA in $(\C_F; \leq_d)$.
\end{remarks}

We note that all of this can be carried out with more general versions (as in \cite{Wa}) of the predimension construction. Non-existence of precompact expansions; non-amenability and failure of EPPA all follow in the same way, assuming minor non-triviality conditions. In particular, Hrushovski's strictly stable $\omega$-categorical structures from \cite{Hrpp} show that the structure $M$ in Theorem~\ref{cex} may also be taken to be stable. By contrast, note that if $M$ is a stable structure which is homogeneous for a finite relational language, then $M$ is $\omega$-stable and, by the strong structure theory of $\omega$-categorical, $\omega$-stable theories, it follows that $\Aut(M)$ is amenable and has a co-precompact, extremely amenable subgroup (cf. Corollary 3.9 in \cite{B}).

\section{Meagre orbits} \label{sec:meagre}

\begin{notation}  In this section, $(\C; \leq)$ will denote one of the amalgamation classes $(\C_0; \leq_s)$ or $(\C_F; \leq_d)$ of $k$-sparse graphs from the previous sections. In the latter case, we assume that $(\C_F; \leq_d)$ is a free amalgamation class and that all vertices are $d$-closed in all structures in $\C_F$.  In both cases, we denote by $(\D; \sqleq)$ the corresponding class of orientations, together with the appropriate notion of closure, $\cl^\sqleq$. So, respectively, $(\D; \sqleq)$ is $(\D_0; \sleq)$ or $(\D_F; \dleq)$, where $\D_F$ is the class of all $k$-orientations of graphs in $\C_F$. We will take $k=2$, to simplify the notation.

We let $M$ denote the Fra\"{\i}ss\'e limit (that is, $M_0$ or $M_F$ respectively) and $G = \Aut(M)$. The space of orientations of $M$ is denoted by $X(\D)$, as in Section~\ref{KPTsec}. 
\end{notation}

The $G$-flow $X(\D)$ is not minimal. Nevertheless, we prove:

\begin{theorem} \label{meagreorbits} With the above notation, if $Y$ is a minimal $G$-subflow of $X(\D)$, then all $G$-orbits on $Y$ are meagre in $Y$. 
\end{theorem}

This is in sharp contrast to what happens when $G$ has a co-precompact extremely amenable subgroup, where every minimal $G$-flow has a comeagre orbit. 

We begin the proof by noting:

\begin{lemma} Suppose $Y$ is a minimal $G$-subflow of $X(\D)$. Then there is $\D' \subseteq \D$ which is a reasonable class of expansions of $(\C; \leq)$ and is such that $Y = X(\D')$. The class $\D'$ has the Expansion Property with respect to $(\C; \leq)$.
\end{lemma}

\begin{proof} The first point  is by Lemma~\ref{reasonable} and the second follows from minimality of $Y$ and Theorem~\ref{EPthm}.
\end{proof}

Now let $\D'$ be as in the above. By Lemma~\ref{waplemma}, the theorem will follow if we show that $(\D'; \leq)$ does not have the weak amalgamation property. We will prove:

\begin{proposition} \label{wapprop} Let $A_0 = \{a\} \in \D'$. Then there does not exist $A_0 \leq A \in \D'$ such that whenever $f_i : A \to C_i \in \D'$ are $\leq$-embeddings (for $i = 1,2$) with $f_1(a) = f_2(a)$, there is $D \in \D'$ and $\leq$-embeddings $g_i : C_i \to D$ such that $g_1(f_1(a)) = g_2(f_2(a))$.
\end{proposition}

We first outline the idea behind the proof of this. Suppose to the contrary that we do have such an $A \in \D'$. Then any finite number of extensions $A \leq C_i \in \D'$ (for $i = 1,\ldots, n$) can be amalgamated \emph{over $A_0$} into some $D \in \D'$.  In particular, the successor-closures of $a$ in the $C_i$ all embed (over $a$) into the successor-closure of $a$ in $D$. We then observe that there are too many possibilities for the successor-closure of $a$ in $C_i$ for this to happen. If $\D = \D'$, then this is very easy to see as we know explicitly the digraphs in $\D'$. In general, we need a result which guarantees that we can `extend' the successor-closure of a point $a \in A$ in particular ways by taking $A \leq C \in \D'$. We do this in a  series of lemmas. The notation is cumulative.

\medskip

Suppose that we have $\{a\} \leq A \in \D'$ contradicting Proposition~\ref{wapprop}. Let $A^- \in \C$ be the undirected reduct of $A$. As $\D'$ satisfies the Expansion Property, there is an extension $A^- \leq B \in \C$ such that every orientation $B^+$ of $B$ in $\D'$ contains a copy of $A$ as a $\leq$-substructure. We fix such a structure $B$.

For sufficiently large $n,m \in \N$, we now describe graphs $T_0(n)$, $T_1(3m)$ (depicted in Figure~\ref{fig:T0}) which we will attach to vertices in $B$ to `extend' the closure of a point in incompatible ways. As $T_0(n)$, we take `one half' of a binary tree of height $n$  together with its root vertex $c$. (For example, we can consider sequences in $\{0,1\}^{< n}$ which are either the empty sequence or start with $0$ and have edge relation given by the initial segment relation.) For $T_1(3m)$, we choose some $m$ such that $\C$ contains an $2m$-cycle (this will be possible for all sufficiently large $m$) and let $T_1(3m)$ be a modification of $T_0(3m)$ obtained by identifying just two vertices at height $2m$ whose shortest paths to the root vertex $c$ meet at height $m$. So in particular, $T_1(3m)$ contains a $2m$-cycle.  Let $T$ denote one of $T_0(n), T_1(n)$ (with $n$ a multiple of 3 in the latter case). Let $T'$ denote the orientation of this in which all edges are directed away from $c$. Note that the only vertices of out-valency less than 2 in $T'$ are the root $c$ and the leaves.
\begin{figure}
\centering
\includegraphics{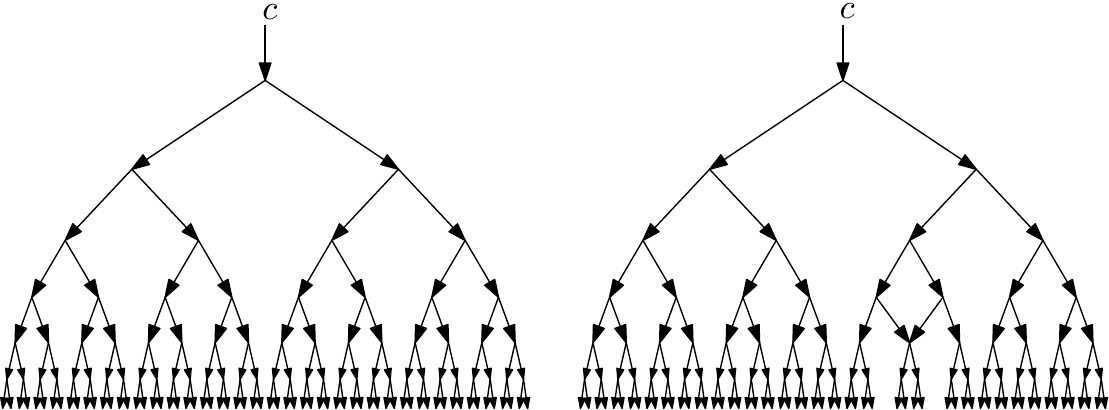}
\caption{$T_0$ and $T_1$ with orientation away from vertex $c$.}
\label{fig:T0}
\end{figure}

Let $S_0$ denote the `left' leaves of $T$ and $S_1$ the `right' leaves. So this is a partition of the leaves of $T$ such that every vertex at height $n-1$ is adjacent to one vertex in each of $S_0$ and $S_1$. 

\begin{lemma} We have $\{c\} \leq T \in \C$ and $S_0, S_1 \leq  \C$.
\end{lemma}

\begin{proof} This is straightforward. Note that in the orientation $T'$ of $T$ the leaves are the only vertices of out-valency less than 2. With this orientation, every non-leaf vertex has a descendant in $S_0$ and in $S_1$, so by Lemma~\ref{lem:sdcl} we obtain $S_i \leq T$ (in the case where $\C = \C_F$). 
\end{proof}

Take some orientation $B^+ \in \D'$ of $B$.  Let $a_1,\ldots, a_r \in B^+$ be all the vertices of out-degree less than $2$ in $B^+$. Denote the outdegrees of $a_i$ by $k_i$. At each vertex $a_i$ we attach $2-k_i$ copies of $T$ to $B$, using free amalgamation identifying $a_i$ and $c$. Call the resulting graph $E$. 

\begin{lemma} We have that $B \leq E \in \C$.
\end{lemma}

\begin{proof}Suppose first that  $\C$ is $\C_0$. Note that the orientation  $B^+$ of $B$ can be extended to an orientation of $E$ in $\D$ so that $B$ is successor-closed (by directing all of the edges in the adjoined copies of $T$ towards the $a_i$). This gives the result. In the case where $\C$ is $\C_F$,  we use the fact that we have chosen $F$ so that vertices are $d$-closed in all graphs in $\C_F$. As $E$ is constructed from $B$ by free amalgamation over vertices, the result follows.
\end{proof}

Now let $S^0$ be the union of the vertices $S_0$ in the copies of $T$ in $E$ which we added to $B$; similarly let $S^1$ be the union of the copies of $S_1$.

\begin{lemma} We have $S^0, S^1 \leq E$.
\end{lemma}

\begin{proof} We can extend the orientation $B^+$ of $B$ to an orientation $E^+ \in \D$ of $E$ so that the added copies of $T$ are directed outwards from the vertices $a_i$. Then $S^j$ is successor-closed in  $E^+$ (for $j=1,2$). So in the case where $\C$ is  $\C_0$ we obtain $S^j \leq E$.

For the case where $\C = \C_F$,  we also note that if $x \in E^+\setminus (S^0\cup S^1)$, then $x$ has a root in each of $S^0$ and $S^1$. In particular, $x$ is not in the successor-$d$-closure of $S^0$ or $S^1$. 
\end{proof} 

We now show that some re-orientation of $E^+ \in \D$ in the above proof which preserves the orientation on the added copies of $T$ is actually in $\D'$. 

Let $E_1$ consist of a sufficiently large number of copies of $E$, freely amalgamated over $S^0$. Note that (by free amalgamation) each of the copies is strong in $E_1$. Let $S^2$ be the the union of the sets of vertices corresponding to $S^1$ in all of these copies of $E$. By a similar argument to that used in the lemma, we have that $S^2 \leq E_1$. 
Now let $P$ consist of the free amalgam of sufficiently many copies of $E_1$ over $S^2$. Again, note that each copy is strong in $P$. 

By construction, $P \in \C$ and therefore it has some orientation $P^+ \in \D'$. In $P^+$, one of the copies $E_1'$ of $E_1$ must be oriented so that the vertices in $S^2$ have no successors in $E_1'$ (as long as we took sufficiently many copies of $E_1$ in $P$). Similarly, there is a copy  $E'$ of $E$ in $E_1'$ in which the vertices in $S^1$ have no successors in $E'$. Thus, in $E'$, the copies of $T$ which were added to the copy $B'$ of $B$ are all directed away from the $a_i$. Moreover, by construction of $P$, we have $E' \leq E_1' \leq P^+$, so $E' \in \D'$. 

Let $B'$ denote the (oriented) copy of $B$ inside $E'$. As $B \leq E$, we have that $B' \in 
\D'$. Thus, by the Expansion Property, we can regard $A$ as a $\leq$-substructure of $B'$.
Note that the vertices $a_i$ in $B'$ have the same out-valency  in $B'$ as they have in $B^+$. By Lemma~\ref{lem:roots} it then follows that the $a_i$ are the only vertices in $B'$ which have out-valency less than 2. As $\delta(A) \geq \delta(a) > 0$, at least one of the vertices $a_i$ is in $\cl^\sqleq_{B'}(a)$.  Without loss of generality, we may assume that these are $a_1,\ldots, a_s$. Recall that $T'$ is the orientation of $T$ where edges are directed away from $c$. We then have:

\begin{lemma}\label{Textn} Let $C = \cl^\sqleq_{E'}(A)$. Then $A \leq C \in \D'$ and $\cl^\sqleq_C(a)$ is the free amalgam of $\cl^\sqleq_{B'}(a)$ and copies $T_1',\ldots, T_s'$ of  $T'$  over $a_1,\ldots , a_s$, for some $s > 0$.  The only vertices of out-valency less than 2 in $\cl^\sqleq_C(a)$ are the leaf-vertices of the $T_i'$. 
\end{lemma}

We can now finish the proof of Proposition~\ref{wapprop} and therefore conclude the proof of Theorem~\ref{meagreorbits}. 

Recall that $T$ was either $T_0(n)$, a binary tree of height $n$, or $T_1(3m)$, a modification of a binary tree of height $3m$  containing a $2m$-cycle with vertices at heights between $m$ and $2m$. We will choose appropriate $n, m$ in what follows. The graphs $A$ and $B$ are the same in both cases and do not depend on $m, n$, though the orientation $B'$ of $B$ in Lemma~\ref{Textn} may do. 

Apply Lemma~\ref{Textn} in the cases $T= T_0(n)$ and $T = T_1(3m)$  to obtain respectively $A\leq C_0 \in \D'$ and $A \leq C_1 \in \D'$ with properties as in the lemma. Suppose that there are $D \in \D'$ and $\leq$-embeddings $g_i : C_i \to D$ (for $i = 0,1$) with $g_0(a) = g_1(a)$. Let $e = g_0(a) = g_1(a)$ and consider the successor-closure $U$ of $e$ in $D$. Let $U_r$ denote the vertices in $U$ which are reachable from $e$ by an outward-directed path of length at most $r$. 

As all vertices in $C_0$ apart from the `leaf vertices' are of out-valency 2,  it follows that $g_0(\scl_{C_0}(a)) \supseteq U_n$: there can be no  vertices in $D$ reachable by a directed path from $e$ of length at most $n$, other than those already in the image of such a path from $a$ in $C_0$. Note that any vertex in $g_0(\scl_{C_0}(a) \cap B)$ is in $U_{q}$, where $q = \vert B \vert$, therefore $U_n \setminus U_{q}$ contains no (undirected) cycles. 

On the other hand, if we take $m \geq  q$, then $g_1(\scl_{C_1}(a)) \setminus U_q$ contains a $2m$-cycle and is contained in $U_{q+3m}$. Thus, if we take $m \geq \vert B \vert$ and $n \geq \vert B \vert + 3m$, then we obtain a contradiction: $U_n \setminus U_q$ contains no cycles, by the previous paragraph. This finishes the proof of Proposition~\ref{wapprop}. Theorem~\ref{meagreorbits} then  follows.

\section{Amenable and extremely amenable subgroups}
\label{sec:C0sec}\label{CFsec}
Recall the class $\ordpair{\C_0;\leq_s}$ of finite 2-sparse
graphs (from Section~\ref{sec:C0}) and the related class $\ordpair{\D_0;\sleq}$
of 2-oriented finite digraphs. As before, the  \Fraisse-limits of these classes are denoted by  $M_0$ and $N_0$ respectively. We also consider the class $\ordpair{\C_F;\leq_d}$ given
in Section~\ref{Homegacat} and the related class $\ordpair{\D_F;\dleq}$ of its
2-orientations. The \Fraisse-limits of these are denoted by $M_F$ and $N_F$. Recall that $\Aut(M_F)$ is oligomorphic.

In previous sections we showed that $\Aut(M_0)$ and $\Aut(M_F)$ have no co-precompact extremely amenable subgroups. Moreover, $\Aut(M_F)$ has no co-precompact amenable subgroup. In this section, using work in \cite{EHN}, we will complement these results by identifying certain closed subgroups which are maximal amongst the extremely amenable subgroups. Roughly speaking, these arise as automorphism groups of ordered versions of $N_0$ and $N_F$ respectively, but in each case, we need to work with a subclass of the class of orientations (the \textit{fine orientations} in Section~\ref{sec:fine}). The precise result is Theorem~\ref{69}.

We also give some partial results about amenable subgroups (Theorem~\ref{611}).

%
%

\subsection{Fine orientations}
\label{sec:fine}

We introduce the following notion of fine orientations.
\begin{definition}
\label{defn:fine}
Suppose that $A, B \in \D_0$ are 2-orientations of the same underlying undirected graph. We say that $B$ is a \emph{refinement} of $A$ (in the class $\ordpair{\D_0; \sleq}$) if every $\sleq$--closed subset of $A$ is also $\sleq$-closed in $B$. The refinement is \emph{proper} if additionally $A$ is not a refinement of $B$. We say that $A$ is \emph{fine} if it has no proper refinement in $(\D_0; \sleq)$. 

Similarly, we can make the same definitions for the class $\ordpair{\D_F; \dleq}$, working with $\dleq$ instead of $\sleq$. In this case, we refer to \emph{$d$-fine} orientations. 
It is clear that every structure in $(\D_0; \sleq)$ (or in $(\D_F; \dleq)$) has a fine refinement (take a refinement with a maximal number of closed subsets).
\end{definition}

We thank the Referee for a simplification to the proof of the following:

\begin{lemma}
\label{lem:finesubstructures}
For every fine $A\in \D_0$ and $B\sleq A$ it holds that $B$ is also fine. Similarly for every $d$-fine $A\in \D_F$ and $B\dleq A$ it holds that $B$ is also fine.
\end{lemma}
\begin{proof}
For the first statement, suppose $B \sleq A$ has a proper refinement $B'$.  Consider the 2-orientation $A'$ created from $A$ by replacing $B$ by $B'$. As the successor-closure operation is unary, it follows that $A'$ is a proper refinement of $A$, which is a contradiction.

\medskip

For the second statement, suppose that $B \dleq A $ has a proper refinement $B'$. We again consider the structure $A'$ obtained from replacing  $B$ by $B'$ in $A$. This is in the class $\D_F$ and it will suffice to show that it is a refinement of $A$. 

Suppose that  $D\dleq A$. Then $D$ is a $d$-closed subset of $A$ and as this is a property of the undirected reduct, we therefore have $D$ is $d$-closed in $A'$. As in the previous case, $D$ is successor-closed in $A'$ and the result follows. 
\end{proof}


We will denote by $\E_0$ the class of all fine orientations in $\D_0$ and by $\E_F$ the class of all $d$-fine orientations in $\D_F$.

\begin{lemma}
\label{lem:fineamalgamation} The class
$\ordpair{\E_0;\sleq}$ is a strong expansion of $\ordpair{\C_0;\leq_s}$,  closed under free
amalgamation. Similarly $\ordpair{\E_F;\dleq}$ is a strong expansion of $\ordpair{\C_F;\leq_d}$
closed under free amalgamation.
\end{lemma}
\begin{proof}
We prove the statements about $\E_F$. The proofs for $\E_0$ are analogous.

By Lemma~\ref{lem:finesubstructures} we know that $\E_F$ is closed under taking $\dleq$-sub\-structures, so $\ordpair{\E_F; \dleq}$ is a strong class. We now verify the conditions (Definition~\ref{stex:def}) for being a strong expansion. Every structure in $\C_F$ has a fine orientation and Condition 
(i) follows by Lemma~\ref{lem:sdcl}, so it remains to verify Condition (ii).

Suppose $A \leq_d B \in \C_F$. Then 
by Theorem~\ref{23} there is an orientation $B^+ \in \D_F$ with $A \dleq B^+$. We may replace $B^+$ here by a fine refinement, and therefore we may assume that $B^+ \in \E_F$. As in the proof of Lemma~\ref{lem:finesubstructures} we can replace in $B^+$ the induced orientation on $A$ by any other fine orientation and the resulting structure is still in $\E^+$. This gives the Condition (ii).

It remains to verify that $\E_F$ is closed under free amalgamation.
Consider the free amalgamation $C$ of $B_1\in \E_F$ and $B_2\in \E_F$
over a common successor-$d$-closed substructure $A$. To show that $C\in \E_F$
it remains to verify that $C$ is $d$-fine. 

Assume, to the contrary, the
existence of a proper refinement $C'$  of $C$  and take $D \dleq C'$ with $D \not\dleq C$.
Put $D_1=D\cap  B_1$ and $D_2 = D \cap B_2$. Because the intersection of two
successor-$d$-closed substructures is successor-$d$-closed, we get that both
$D_1$ and $D_2$ are successor-$d$-closed in $C'$ and because $B_1$ and $B_2$ are fine,
they are also closed in $C$. Now every vertex $v$ in $\sdcl_C(D) \setminus D$
is connected by a directed path in $C$ to some roots of $C$  in $B_1 \setminus B_2$ and in  $B_2 \setminus B_1$. But this implies that $v\in B_1\cap B_2$. This contradicts $B_1$ and $B_2$ being fine.
\end{proof}

\subsection{Closure reducts}
\label{sec:G0sec}

We will determine some amenable and extremely amenable subgroups of $\Aut(M_0)$ and $\Aut(M_F)$; in some cases proving their maximality with respect to these properties. These will be associated with certain (fine) orientations of the structures $M_0$ and $M_F$. However, for the maximality, we will have to pass to  reducts of the oriented structures which remember only  the closures associated to the orientations. In order to do this, we  use partial functions and the notion of substructure introduced
in Section~\ref{rcsec}. This will enable us to apply directly the results of \cite{EHN}.

\begin{definition}
(1) Suppose $\ordpair{A; S} \in \D_0$. Denote by $A^\circ = \ordpair{A;\allowbreak R;\allowbreak (F_k)_{1 \leq k}}$
the following structure in the language consisting of one binary relation $R$ and partial
functions $F_k$ from vertices to sets of vertices of size $k$. The relation $R$ is the symmetrised (i.e. undirected)
 $S$ and for a vertex $a$ with closure $C = \scl_A(\{a\})$ we put
$F_{\vert C\vert }(a) = C$. The functions $F_k$ are undefined otherwise.

(2) Suppose $\ordpair{A; S} \in \D_F$. We  denote by $A^\bullet = \ordpair{A;\allowbreak R;\allowbreak (F_k)_{1 \leq k},\allowbreak (F_{k,n})_{1 \leq n, k}}$ the following structure in the language consisting of one binary relation $R$ and
partial functions $F_k$ from vertices to sets of vertices of size $k$ and partial functions $F_{k,n}$ from $n$-tuples of vertices to sets of vertices of size $k$. The relation $R$ is the symmetrised $S$ and for a vertex $a$ with closure $C = \scl_A(\{a\})$ we put $F_{\vert C\vert }(a) = C$.  For an $n$-tuple $\vec{v}$ of distinct root vertices in $A$ 
we put $F_{k,n} (\vec{v}) = U$ where $\vert U \vert = k$ and $U$ is the set of all vertices $u$ with the property that $\vec{v}$ consists precisely of the roots of $\scl_A(\{u\})$. The functions $F_k$ and $F_{k,n}$ are undefined otherwise.
\end{definition}
Note that we use successor-closure rather than successor-$d$-closure in the definition of $A^\bullet$. 
Recall that $\subseteq$ is inclusion and by $(\K;\subseteq)$ we denote strong classes where strong maps are all embeddings.
By the definition of successor-closed and successor-$d$-closed substructures (Definition~\ref{def:sdcl}) we obtain:
\begin{lemma}
\label{lem:fineamalg}
For all $B \subseteq A\in \D_0$, it holds that $B\sleq A$ if and only if the vertices of $B$ form a substructure of $A^\circ$.

Similarly for all $B\subseteq A\in \D_F$, it holds that $B\dleq A$ if and only if  $B$ is a substructure of $A^\bullet$.
\end{lemma}

We will denote by $\G_0$ the class of all structures $A^\circ$ where $A \in \E_0$, and by
$\G_F$ the class of all structures $A^\bullet$, where  $A \in \E_F$.

Observe that there is important difference between $\G_0$ and $\G_F$. While
all functions in $\G_0$ are unary, the functions in $\G_F$ have arbitrary arities. 

Recall the notion of a free amalgamation class with respect to $\subseteq$ (all embeddings), introduced prior to Theorem~\ref{thm:Ramseyclosures}. We have:

\begin{lemma}
\label{lem:Gamalg}
The class $\ordpair{\G_0;\subseteq}$ is a strong expansion of $\ordpair{\C_0;\leq_s}$.

Similarly $\ordpair{\G_F;\subseteq}$ is a strong expansion of $\ordpair{\C_F;\leq_d}$. Both classes are free amalgamation classes.
\end{lemma}

\begin{proof}
This follows by Lemmas~\ref{lem:fineamalgamation}  and \ref{lem:fineamalg}. For the statement about free amalgamation, one checks, in the case of $\ordpair{\G_F;\subseteq}$, that if $A \dleq B_1, B_2 \in \E_F$ and $C$ is the free amalgam of $B_1$ and $B_2$ over $A$ (in $\E_F$), then $C^\bullet$ is the free amalgam of $B_1^\bullet$ and $B_2^\bullet$ over $A^\bullet$ (in the sense of Theorem~\ref{thm:Ramseyclosures}). The proof for $\ordpair{\G_0;\subseteq}$ is similar, but more straightforward.
\end{proof}

\subsection{Extremely amenable subgroups}
\label{sec:DFramsey}
\begin{theorem}
\label{thm:expansion}
\label{thm:Fexpansion}
\label{thm:GRamsey}
The class $\ordpair{\G_0^\prec;\subseteq}$ of linear orderings of structures in $\ordpair{\G_0;\subseteq}$
is Ramsey and moreover it has a subclass $\ordpair{\mathcal H_0;\subseteq}$ which is a Ramsey expansion of $\ordpair{\G_0;\subseteq}$ having the expansion property (with respect to $\ordpair{\G_0;\subseteq}$).

The class $\ordpair{\G_F^\prec;\subseteq}$ of linear orderings of structures in $\ordpair{\G_F;\subseteq}$
is Ramsey and moreover it has a subclass $\ordpair{\mathcal H_F;\subseteq}$ which is a Ramsey expansion of $\ordpair{\G_F;\subseteq}$ having the expansion property (with respect to $\ordpair{\G_F;\subseteq}$).
\end{theorem}
\begin{proof}
This is a direct consequence of Lemma~\ref{lem:Gamalg} and Theorem~\ref{thm:Ramseyclosures}.
\end{proof}

\begin{remark}In \cite{EHN}, subclasses $\ordpair{\mathcal H_0;\subseteq}$ and $\ordpair{\mathcal H_F;\subseteq}$
are explicitly described by means of special, admissible, orderings. The 
Ramsey expansion of $\ordpair{\G_F ; \subseteq}$ is more involved, due to the presence of non-unary functions in the structures, and needs the full power
of \cite{HN}.
\end{remark}
With the notation as in Theorem~\ref{thm:GRamsey}, we  will denote by  $P_0$ the \Fraisse{} limit of $\ordpair{\mathcal H_0 ; \subseteq}$ and $P_F$ the \Fraisse{} limit of $\ordpair{\mathcal H_F ; \subseteq}$. We also denote by  $O_0$ is the \Fraisse{} limit of $\ordpair{\G_0; \subseteq}$ and by $O_F$ is the \Fraisse{} limit of $\ordpair{\G_F; \subseteq}$. As usual, we can regard $P_0$ as an expansion of $O_0$ and $P_F$ as an expansion of $O_F$. Note that these are both precompact expansions. Furthermore, if we let $E_0$ and $E_F$ be the \Fraisse{} limits of $\E_0$ and $\E_F$, then we may regard $O_0$ and $O_F$ as the closure-reducts $E_0^\circ$ and $E_F^\bullet$ of $E_0$ and $E_F$. 
\begin{theorem} 
\label{59}
\label{69}
The subgroup $\Aut(P_0)$ is maximal among extremely amenable subgroups of $\Aut(M_0)$. Similarly $\Aut(P_F)$ is maximal among extremely amenable subgroups of $\Aut(M_F)$.
\end{theorem}
\begin{proof}
Again we give the proof for $M_F$ while the statement for $M_0$ follows in complete analogy. Extreme amenability of $\Aut(P_F)$ follows by Theorems~\ref{thm:GRamsey} and~\ref{thm:KPT}.

Suppose $\Aut(P_F ) \leq K \leq \Aut(M_F )$ and $K$ is extremely amenable.
We show that $K = \Aut(P_F )$.

Consider $K$ acting on $X(\D_F)$, the space of 2-orientations of $M_F$. As
$K$ is extremely amenable, it preserves some element of this. Thus $K$ is
acting as a group of automorphisms of some 2-orientation $N$ of $M_F$. Consider any finite closed  $A \subseteq P_F$ (or equivalently, $A \dleq E_F$). If $b \notin
  A$ then by the homogeneity of $P_F$, the orbit of $b$ under the pointwise stabiliser
of $A$ in $\Aut(P_F)$ is infinite. So $b$ is not in the $\dleq$-closure of $A$ in $N^\bullet$. It
follows that $A\dleq N$. As $E_F$ is a fine orientation of $M_F$, it follows that $\dleq$-closure is the same in $E_F$ and in $N$. So in particular, $K \leq \Aut(O_F)$.

Thus both $K$ and $\Aut(P_F)$ are co-precompact, extremely amenable subgroups of
$\Aut(O_F)$. By Theorem~\ref{thm:GRamsey} and Theorem~\ref{thm:comeagre} it follows that $K \leq \Aut(P_F)$.
\end{proof}

\subsection{Amenable subgroups}

The following is a simple generalisation of Proposition 9.3 of \cite{AK}.

\begin{theorem} \label{thm:easyamen} Suppose $\ordpair{\K; \subseteq}$ is a free amalgamation class of structures in a language with relations and partial functions. Let $O$ be the \Fraisse{} limit of $\ordpair{\K; \subseteq}$. Then $G = \Aut(O)$ is amenable.
\end{theorem}

\begin{proof} By Theorem~\ref{thm:Ramseyclosures}, the class $\ordpair{\K^\prec; \subseteq}$ of all linear orderings is a Ramsey class. Let $P$ be its \Fraisse{} limit. Then $H = \Aut(P)$ is a co-precompact extremely amenable subgroup of $G$ and the completion $\widehat{G/H}$ is isomorphic (as a $G$-flow) to the space of linear orderings on $O$. This has a $G$-invariant probability measure. Moreover, as $H$ is extremely amenable and the $G$-orbit $G/H$ is dense in $\widehat{G/H}$, there is a continuous $G$-map from $\widehat{G/H}$ to any other $G$-flow.
\end{proof}

Recall the following notation:
\begin{enumerate}
\item $N_0$ is the \Fraisse{} limit of $\ordpair{\D_0; \sleq}$;
\item $E_0$ is the \Fraisse{} limit of the fine orientations $\ordpair{\E_0; \sleq}$;
\item $O_0$ is the \Fraisse{} limit of the closure-reducts $\ordpair{\G_0; \subseteq}$;
\item $N_F$ is the \Fraisse{} limit of $\ordpair{\D_F; \dleq}$;
\item $E_F$ is the \Fraisse{} limit of the $d$-fine orientations $\ordpair{\E_F;\dleq}$;
\item $O_F$ is the \Fraisse{} limit of their closure reducts $\ordpair{\G_F; \subseteq}$.
\end{enumerate}

Each of these, and their closure reducts may be regarded as the \Fraisse{} limit of a free amalgamation class $\ordpair{\K; \subseteq}$ in a language with relations and partial functions. Thus by Theorem~\ref{thm:easyamen} we have:

\begin{theorem} \label{611} The groups $\Aut(N_0)$, $\Aut(N_0^\circ)$, $\Aut(E_0)$, $\Aut(O_0)$, $\Aut(N_F)$, $\Aut(N_F^\bullet)$, $\Aut(E_F)$, $\Aut(O_F)$ are amenable.
\end{theorem}

It would be interesting to know whether the subgroup $\Aut(O_0)$ is maximal amenable subgroup of $\Aut(M_0)$ and whether $\Aut(O_F)$ is maximal amenable subgroup of $\Aut(M_F)$.

\section{Concluding remarks}

It would of course be interesting to find other types of counterexamples for Question~\ref{qu}. 

As already mentioned, the following question, raised by a number of authors, remains open:

\begin{question} \rm Suppose $M$ is a countable structure which is homogeneous in a finite relational language. Does there exist an $\omega$-categorical expansion $N$ of $M$ whose automorphism group is extremely amenable? If so, can we take $N$ homogeneous in a finite relational language?
\end{question}

 In view of Theorem~\ref{sparsethm}, we ask the following:
 
 \begin{question} \rm Does there exist a countable structure $M$ which is homogeneous in a finite relational language and in which a sparse graph of infinite valency can be interpreted?
 \end{question}

By Theorem~\ref{sparsethm}, such an $M$ would give a negative answer to the first question. 

It is worth remarking that an $\omega$-categorical  sparse graph of infinite valency interprets a \emph{pseudoplane}: a bipartite graph where all vertices are of infinite valency and in which any pair of vertices have only finitely many common neighbours. The Hrushovski construction in Section~\ref{hcsec} is the only way of producing $\omega$-categorical pseudoplanes currently known. Moreover, it is an open question whether there is a pseudoplane which is homogeneous in a finite relational language (see \cite{ThomasPP}).

Note that in our examples of groups $G$ for Question~\ref{qu}, the obstacle to having a precompact extremely amenable subgroup, that is, the $G$-flow of orientations,  is also the obstacle to $G$ being amenable. So the following version of Question~\ref{qu} considered by Ivanov in \cite{Ivanov} is particularly interesting:

\begin{question} Suppose $M$ is a countable $\omega$-categorical structure with amenable automorphism group. Does there exist an $\omega$-categorical expansion of $M$ whose automorphism group is extremely amenable?
\end{question}

The groups $\Aut(M_0)$ and $\Aut(M_F)$ have metrizable minimal flows all of whose orbits are meagre (Theorem \ref{meagreorbits}). The following has been raised by T. Tsankov (personal communication):

\begin{question} Does either of these groups have a non-trivial, (metrizable) minimal flow with a comeagre orbit?
\end{question}

We note that the paper \cite{OlaK} by A. Kwiatkowska has an example of a countable structure which is not $\omega$-categorical, but whose automorphism group $G$ is Roelcke precompact and is such that $M(G)$ is non-metrizable and has a comeagre orbit.

In Theorem~\ref{thm:easyamen} we gave an easy proof of amenability (modulo a hard Ramsey result) which did not use EPPA. However, it would still be interesting to have the EPPA results as they imply other properties of automorphism groups, beyond amenability. 
By Theorem~1.7 of \cite{EHN} it follows that class $\G_0$ has EPPA.
Because the language of $\G_F$ contains non-unary functions, we cannot establish the expansion property for partial automorphisms by application of \cite{EHN}. However, it  seems reasonable to conjecture it. Moreover, we also believe that some of the amenable subgroups are maximal amenable. In particular, we propose:
\begin{conjecture}\rm
The class $\ordpair{\G_F ; \subseteq}$ has EPPA. Moreover, if $\Aut(O_F) \leq H \leq \Aut(M_F)$ and $H$ is amenable, then $H \leq \Aut(O_F)$.
\end{conjecture}

Our results can also be formulated in the context of \cite{AK}. One can easily see (implicitly in \cite{AK}) that the concepts of excellent pair $(\mathcal K,\mathcal K^*)$ and consistent random $\mathcal K$-admissible orderings to classes endowed with strong mappings and (more complicated)  expansions, thus obtaining results analogous to Proposition 9.2 of \cite{AK}. 
This is interesting if we apply this to ordering theorems for substructures (i.e. to canonical orderings) which were in this context studied in~\cite{NRO}, see also \cite{BJ} for related research.

\end{document}